\documentclass{amsart}
\usepackage{latexsym,amsxtra,amscd,ifthen}
\usepackage{amsfonts}
\usepackage{pifont}
\usepackage{verbatim}
\usepackage{amsmath}
\usepackage{graphicx}
\usepackage{amsthm}
\usepackage{amssymb}
\usepackage{url,color}
\usepackage[arrow,matrix]{xy}
\usepackage{pict2e,epic,eepic,pstricks}

\usepackage{anysize}\marginsize{25mm}{25mm}{30mm}{38mm}
\allowdisplaybreaks[4]
\allowdisplaybreaks[4]

\newtheorem{thm}{Theorem}[section]

\newtheorem{cor}[thm]{Corollary}
\newtheorem{lem}[thm]{Lemma}

\newtheorem{prop}[thm]{Proposition}
\newtheorem{prop-def}[thm]{Proposition-Definition}
\theoremstyle{definition}
\newtheorem{Def}[thm]{Definition}
\theoremstyle{remark}
\newtheorem{rmk}[thm]{\bf Remark}
\newtheorem{exm}[thm]{\bf Example}
\numberwithin{equation}{section}

\usepackage{bbm}
\usepackage{float}

\newcommand{\ucr}[1]{\underset{#1}{\circ}}
\newcommand{\dbar}[1]{\overline{\overline{#1}}}
\newcommand{\bc}[1]{\underset{#1}{\boxtimes}}

\def \f {\varphi}
\def \Coder{\operatorname{Coder}}	
\def\N{\mathbb{N}}
\def\Sum{\sum\limits}
\def\la{\lambda}
\def \ot{\otimes}

\def \I{\mathcal{I}}

\def \ip{{\mathcal{P}}}

\def \Hom{\operatorname{Hom}}
\def \End{\mathcal{E}\mathsf{n}\mathsf{d}}

\def \Ker{\operatorname{Ker}}
\def \Id{\operatorname{Id}}
\def \1{\mathbbm 1}

\def \Z{\mathbb{Z}}
\def\al{\alpha}
\def\vert{\operatorname{Vert}}
\def\int{\operatorname{Int}}
\def\leav{\operatorname{Lv}}
\def\pnt{\operatorname{Parent}}
\def\Mod{\mbox{-{\rm Mod}}}
\def\lb{\operatorname{x}}
\def\mx{\mathcal{X}}
\def\ias{\mathcal{A}\mathsf{s}}
\def\C{\operatorname{C}}
\def\proj{\operatorname{proj}}
\def\Der{\operatorname{Der}}
\def\art{\operatorname{art}}
\def\M{M}
\def\F{\mathcal{F}}
\def \h {\operatorname{Hoch}}
\def \Ext {\operatorname{Ext}}
\def \1{\mathbf{1}}

\def \ka{\operatorname{K}}

\def \As{\mathcal{A}\mathsf{g}\mathsf{d}}
\def \Agd{\mathcal{A}\mathsf{g}\mathsf{d}}
\def\ic{\mathcal{C}}
\def \c{\circ}
\def \ra{\rightarrow}
\def\ac{{}^{\scriptstyle \textrm{\rm !`}}}

\def \ob{\mathrm{Ob}}
\def \d{\delta}
\def \gder{\operatorname{AsGDer}}
\def\Coder{\operatorname{Coder}}

\begin{document}
\title{Operad and cohomology of associative algebras with generalized derivation}

\author{Jiang-Nan Xu}
\address{(Xu) School of Mathematical Sciences, Anhui University, Hefei 230601, China}
\email{jnxu1@outlook.com}

\author{Yan-Hong Bao}
\address{(Bao) School of Mathematical Sciences, Anhui University, Hefei 230601, China}
\email{baoyh@ahu.edu.cn}




\begin{abstract}
An associative algebra with a generalized derivation is called an AsGDer triple. 
We introduce the operad that encodes AsGDer triples, and prove it is a Koszul operad. 
Using its Koszul dual cooperad, we introduce the homotopy version of AsGDer triples.
As an application, we construct the AsGDer cohomology theory for AsGDer triples,
and show that the formal deformation of an AsGDer triple is controlled by the AsGDer cohomology. 
\end{abstract}

\maketitle

\setcounter{section}{-1}
\section{Introduction}
\label{sec0}

Derivations are a fundamental invariant for various kinds of algebras,
which is closely related to cohomology theory, deformation theory, high ordered structure and so on.
As applications of algebra, derivations play an important role in differential geometry, 
control theory and gauge theories in quantum field theory.

In this paper, we mainly study associative algebra with a generalized derivation.
The notion of generalized derivations was firstly introduced by Bre$\check{\rm s}$ar in \cite{Br},
which was used to study a certain generalization of Posner's theorem \cite{Po}. 
Let $A$ be an associative algebra and $h$ a derivation of $A$. 
Recall that a linear map $\delta\colon A\to A$ is called a \textit{generalized $h$-derivation} (or \textit{generalized derivation} for short)
if $\delta(ab)=\delta(a)b+ah(b)$ for all $a, b\in A$. 
Such a triple $(A, h, \delta)$ of an associative algebra $A$ with  a derivation $h$ and a generalized $h$-derivation $\delta$ is called an AsGDer triple. 
Clearly, the notion of generalized derivations covers the notions of derivations and left multipliers (that is, a linear map $\varphi\colon A \to A$
satisfying $\varphi(ab)=\varphi(a)b$ for all $a, b$).
In fact, any generalized derivation can be written as a sum of left multiplier and a derivation. 
The algebraic properties of generalized derivations have been extensively investigated by numerous researchers, 
see \cite{AC, BO,HD,Hv,Le,Liu,LS, Na1, QKR},
and the concept of generalized derivations was further developed to encompass (e.g. higher order derivations) \cite{Na2}
and nonassociative settings \cite{LL}.
In \cite{HMPS}, the authors developed differential geometry (connection, curvature, etc.) based on generalized derivations.

Algebras with derivations are studied from operadic viewpoint, \cite{DL, Lo}. 
Operad theory originated from algebraic topology and have been widely used in algebra, combinatorics, geometry, mathematical 
physics and topology \cite{LV,Gi,MSS}. 
Operad is an algebraic object that models the operations with multiple inputs and one output.
An algebra of a certain type is usually defined by generating operations and relations, and they can be encoded by operads.
For example, associative algebras, Lie algebras and Poisson algebras are encoded by the operad $\mathcal{A}\mathsf{s}$, $\mathcal{L}\mathsf{ie}$ and $\mathcal{P}\mathsf{ois}$ respectively.  
This  motivates us to use operad theory to study AsGDer triples.

\begin{Def}\label{0-def-operad Agd}
	Let $(E, R)$ be a quadratic data, where $E$ is an $\N$-graded vector space with $E(1)=\Bbbk h\oplus \Bbbk \delta$, $E(2)=\Bbbk \mu$ and $E(n)=0$ for all $n\neq 1, 2$, and $R$ is a sub-$\N$-graded vector space of $\F(E)$ spanned by 
	\begin{align*}
		\mu\ucr 1\mu- \mu\ucr 2\mu, \quad 
		h\ucr 1 \mu -\mu\ucr 1 h-\mu\ucr 2 h\quad {\rm and} \quad 
		\delta\ucr 1 \mu - \mu\ucr 1 \delta-\mu\ucr 2 h.	
	\end{align*}
	Then the quadratic operad associated to the quadratic data $(E, R)$ is called \textit{AsGDer operad}, denoted by $\Agd$.
\end{Def}

Koszul duality theory was generalized from associative algebras to operads  in \cite{GK,GJ}. 
By using the rewriting method, we prove that the operad $\Agd$ is Koszul, see Theorem \ref{2-thm-Adg is Koszul}.
By direct computing, we obtain the explicit description of the Koszul dual operad of $\Agd$.

\begin{prop}\label{0-prop-def of koszul dual operad of Agd}
	Let $\Agd^!$ be the Koszul dual operad of the AsGDer operad $\Agd$. Then 
	\[\As^!(n)= \Bbbk\bar\mu_n \oplus \Bbbk \bar h_n\oplus\Bbbk \bar\delta_n\]
	with $|\bar\mu_{n}|=0$ and $|\bar h_n|=|\bar\delta_n|=-1$ for $n\ge 1$, $\mu_1\in \Agd(1)$ is the identity, and the partial composition is given by
	\begin{equation*}
		\begin{array}{lll}
			\bar\mu_n\ucr i\bar\mu_m=\bar\mu_{m+n-1}, & \bar h_n\ucr i\bar\mu_m=\bar h_{m+n-1}, & \bar\delta_n\ucr i\bar\mu_m=\bar\delta_{m+n-1},\\
			\bar\mu_n\ucr i \bar h_m=\begin{cases}
				\bar h_{m+n-1}, &i=1,\\
				\bar\delta_{m+n-1}+\bar h_{m+n-1}, & 2\leq i\leq n,
			\end{cases} & \bar h_n\ucr i\bar h_m=0, & \bar\delta_n\ucr i\bar h_m=0,\\
			\bar\mu_n\ucr i\bar\delta_m=\begin{cases}
				\bar\delta_{m+n-1}, &i=1,\\
				0, &2\leq i\leq n,
			\end{cases} & \bar h_n\ucr i\bar\delta_m=0, & \bar\delta_n\ucr i\bar\delta_m=0.
		\end{array}
	\end{equation*}
\end{prop}  

When an algebra is encoded by a Koszul operad,  there are  general ways to construct the cohomology complex and the corresponding higher homotopy structure \cite{Ba1,Ba2, LV}. Firstly, we give an explicit expression of Koszul dual cooperad $\Agd^{\ac}$.

\begin{prop}\label{0-prop-def of Koszul dual cooperad }
	Let $\Agd^{\ac}$ be the Koszul dual cooperad of $\Agd$. Then
	\[
	\Agd^{\ac}(n)= \Bbbk \mu^c_n\oplus \Bbbk h^c_n \oplus \Bbbk  \d^c_n \quad ({\rm for} ~ n\geq 1), 	
	\]
	with $|\mu^c_n|=n-1$, $|h^c_n|=|\d^c_n|=n$, the counit $\varepsilon$ is given by $\varepsilon(\mu^c_1)=\1\in \I$,
	and the infinitesimal decomposition is given by
	\begin{align*}
		\Delta_{(1)}( \mu^c_n)=&\sum_{p+q+r=n}(-1)^{r(q-1)}\mu^c_{p+r+1}\bc {p+1}\mu^c_{q},\\
		\Delta_{(1)}(h^c_n)=&\sum_{p+q+r=n}
		(-1)^{r(q-1)}\mu^c_{p+r+1}\bc {p+1} h^c_{q}
		+\sum_{p+q+r=n}(-1)^{(r+1)(q-1)} h^c_{p+r+1}\bc {p+1}\mu^c_{q},\\	
		\Delta_{(1)}( \d^c_n)=& \sum_{p+q+r=n}(-1)^{(r+1)(q-1)} \d^c_{p+r+1}\bc {p+1}\mu^c_{q}
		+\sum_{q+r=n}(-1)^{r(q-1)}\mu^c_{r+1}\bc {1}\d^c_{q}\\
		&+\sum_{p+q+r=n\atop p\geq1}(-1)^{r(q-1)}\mu^c_{p+r+1}\bc {p+1} h^c_{q}.
	\end{align*}
\end{prop}

Similar to $A_\infty$-algebras and $L_\infty$-algebras,  we introduce the definition of homotopy AsGDer triples,
and show it coincides with that induced by a square-zero coderivation on the cofree $\As^{\ac}$-coalgebra.

\begin{Def}
	Let $A$ be a graded vector space. An $\Agd_{\infty}$-\textit{algebra structure} on $A$ consists of collections of operations $\{m_n,h_n,\delta_n|n\geq1\}$ with  
	\[m_n\colon A^{\ot n}\ra A,~ |m_n|=n-2,\quad
	h_n\colon A^{\ot n}\ra A,~ |h_n|=n-1,\quad
	\d_n\colon A^{\ot n}\ra A,~ |\d_n|=n-1,\]
	satisfying the following relations
	\begin{align*}
		&\sum_{p+q+r=n}(-1)^{p+qr}m_{p+r+1}\ucr {p+1}m_{q}=0,\\
		&\sum_{p+q+r=n}(-1)^{p+qr}m_{p+r+1}\ucr {p+1}h_{q}
		+\sum_{p+q+r=n}(-1)^{p+q(r+1)}h_{p+r+1}\ucr {p+1}m_{q}=0,\\	
		&\sum_{p+q+r=n\atop p\geq1}(-1)^{p+qr}m_{p+r+1}\ucr {p+1}h_{q}
		+\sum_{q+r=n}(-1)^{qr}m_{r+1}\ucr {1}\delta_{q}\\
		&+\sum_{p+q+r=n}(-1)^{p+q(r+1)}\d_{p+r+1}\ucr {p+1}m_{q}=0
	\end{align*}
	for $n\geq1$.
	The quadruple $(A, m_n, h_n , \delta_n)$ is called an \textit{$\Agd_\infty$-algebra}, or a \textit{homotopy AsGDer triple}.
\end{Def}

\begin{thm}\label{0-thm-Agd_infty str=sq zero coder}
	Let $A$ be a graded vector space.	An $\As_{\infty}$-algebra structure on  $A$ is equivalent to a square-zero coderivation on the cofree $\As^{\ac}$-coalgebra $\As^{\ac}(A)$.
\end{thm}

Cohomology theories for various kinds of algebras have been investigated by many mathematicians.
The deformation theory of algebraic structures began with the famous work of Gerstenhaber for associative algebra.
Recently, the simultaneous deformation of algebraic operation and their additional structure, such as derivations and algebra morphisms,
are widely studied. 
For examples, Lie algebras with derivations are studied from cohomological viewpoint in \cite{TFS}. 
Associative algebras and Leibniz algebras with a derivation were studied in \cite{DM, Da}. 

By using the graded Lie algebra structure on the cofree coalgebra over the Koszul dual cooperad $\Agd^{\ac}$, 
we introduce the AsGDer-complex and show that it controls the formal deformation
of AsGDer triples.
\begin{Def}
Let $(A, h, \delta)$ be an AsGDer triple and $\mu$ the multiplication of $A$. 
 Denote the spaces
\begin{itemize}
	\item $\C^n_{\gder}(A, A)\colon =0$ for $n\le 0$,
	\item ${\C}^1_{\gder}(A, A)\colon=\Hom(A,A)$,
	\item $C^n_{\gder}(A, A)\colon=\Hom(A^{\otimes n},A)\oplus\Hom(A^{\otimes,n-1},A)\oplus\Hom(A^{\otimes, n-1},A)$ for $n>1$,
\end{itemize} 
and the linear map 
\begin{align*}
	\partial^n\colon C^n_{\gder}(A,A)&\to C^{n+1}_{\gder}(A,A)\\
	\partial^n(f)\colon&=(-1)^{n-1}[(\mu,h,\delta),~f]_{L_{\Agd}(A)}
\end{align*}
where $L_{\Agd}(A)$ is a graded Lie algebra with the graded Lie bracket induced by one of the graded Lie algebra $\Coder(\Agd^{\ac}(A))$. 
Then the complex $(\C_{\gder}^\bullet(A, A), \partial^\bullet)$ is called the \textit{AsGDer complex} for the AsGDer triple $(A, h, \delta)$, 
and its $n$-th cohomology group is called the \textit{$n$-th AsGDer cohomology group} of $(A, h, \delta)$, denoted by $H^n_{\gder}(A, A)$, for $n \geq 1$.
\end{Def}

\begin{thm}\label{0-thm-deformation thm}
Let $(A, h, \delta)$ be an AsGDer triple and $H_{\gder}^n(A, A)$ be the $n$-th AsGDer cohomology group of $(A, h, \delta)$.
\begin{enumerate}
	\item[(i)] The triple $(\mu_1, h_1, \delta_1)$ is an infinitesimal deformation of $(A, h, \delta)$ if and only if 
	$(\mu_1, h_1, \delta_1)$ is a $2$-cocycle in the AsGDer complex of $A$. In particular, if $H_{\gder}^2(A, A)=0$,
	then $(A, h, \delta)$ has no nontrivial formal deformation.
	
	\item[(ii)] For each $n\ge 2$, if the deformation equations {\rm (DE$_k$-1) }, {\rm (DE$_k$-2) } and {\rm (DE$_k$-3) } hold for $k=1, \cdots, n-1$,
	then $(\ob'_n, \ob''_n, \ob'''_n)$ is a $3$-cocycle in the AsGDer complex, and hence $H_{\gder}^3(A, A)$ is the obstruction cocycle. In particular,
	if $H_{\gder}^3(A, A)=0$, then all the obstructions vanish.
\end{enumerate}	
\end{thm}

The paper is organized as follows. 
In Section 1, we recall some basic notions and results about (co)operads and generalized derivations.
In Section 2, we first construct the operad $\Agd$ that encodes the category of AsGDer triples, 
and show that $\Agd$ is a Koszul operad. The Koszul dual operad $\Agd^!$ and the free $\Agd$-algebras are calculated explicitly. 
In Section 3,  we give an explicit description of the Koszul dual cooperad $\Agd^{\ac}$,
and use it to introduce  the definition of homotopy $\Agd$-algebras.
In Section 4, we construct the cohomology complex for an AsGDer triple by using the cooperad $\Agd^{\ac}$,
and show that it controls the formal deformation of an AsGDer triple. 

Throughout this paper, $\Bbbk$ is a fixed field of characteristic
zero, and (co)operads are $\Bbbk$-linear. All unadorned
$\ot$ will be $\ot_\Bbbk$ and  all unadorned
$\Hom$ will be $\Hom_\Bbbk$. For an element $\mu$ in a graded vector space, $|\mu|$ denotes the degree of $\mu$.

\noindent\textbf{Acknowledgments}.
Y.-H. Bao was partially 
supported by the National Natural Science Foundation of 
China (Grant No. 12371015) and the Science Fundation for 
Distinguished Young Scholars of Anhui Province (No. 2108085J01).

\section{Preliminaries}
\label{sec1}

In this section, we recall some notions, notations and facts about (co)operads and generalized derivations of  associative algebras. For more details about operads, cooperads and Koszul duality, see \cite{LV}.  

\subsection{Operads and cooperads}
Let $\N\Mod$ be the category of $\N$-graded vector spaces. 
For $\N$-graded vector spaces $\mathcal{P}=\{\mathcal{P}(n)\}_{n\ge 0}$ and $\mathcal{Q}=\{\mathcal{Q}(n)\}_{n\ge 0}$, the composition product $\mathcal{P}\circ \mathcal{Q}$ is given by 
\[
(\mathcal{P}\circ \mathcal{Q})(n):=\bigoplus_{k\ge 1} \mathcal{P}(k)\otimes \left(\bigoplus_{i_1+\cdots+i_k=n} \mathcal{Q}(i_1)\otimes\cdots \otimes \mathcal{Q}(i_k)\right).
\]
Note that the composition product satisfies the associativity, and the triple $(\N\Mod, \circ, \I)$ form a monoidal category, where 
$\I$ is the unit object with $\I(1)=\Bbbk$ and $\I(n)=0$ for all $n\neq 1$. 
For an element $\theta\in \ip(n)$, we say that the arity of $\theta$ is $n$, denoted by $\art(\theta)=n$. 

\begin{Def}\label{def-ns operad}
A \textit{nonsymmetric operad} $\ip$ is a monoid $P$ in the monoidal category $(\N\Mod, \circ, \I)$, that is, there is a so-called \textit{composition map} $\gamma\colon \ip\circ \ip\to \ip$ and a \textit{unit map} $\eta\colon\I\to \ip$ satisfying the associativity and the unitality.

A \textit{nonsymmetric cooperad} is a comonoid $\ic$ in the monoidal category $(\N\Mod, \circ, \I)$, that is, there is a \textit{decomposition map} $\Delta\colon \ic\to \ic\circ \ic$, and a {counit} map $\varepsilon\colon \ic\to \I$ satisfying coassociativity and counitality.
\end{Def} 
A nonsymmetric operad (ns operad for short) is also called a \textit{non-$\Sigma$ operad} or \textit{plain operad}. 
In this paper,  we will also use  the partial definition of a ns 
operad:

\begin{Def}	\label{xxdef-partial def operad}
A \textit{nonsymmetric operad} consists of the following data:
\begin{enumerate}
\item[(i)]
an $\N$-graded vector space $(\ip(n))_{n\geq 0}$, whose elements are
called \textit{$n$-ary operations},
\item[(ii)]
an element $\1\in \ip(1)$, called the \textit{identity},
\item[(iii)]
for all integers $m\ge 1$, $n \ge 0$, and $1\le i\le m$, a partial composition map
\[-\ucr{i}-\colon \ip(m) \otimes \ip(n) \to \ip(m+n-1), \]
\end{enumerate}
satisfying the following axioms:
\begin{enumerate}
\item[(OP1)] (Identity) 		
for any $\theta\in \ip(n)$ and $1\leq i\leq n$,
\[\theta\ucr{i}\1 = \theta =\1\ucr{1}\theta;
\]		
\item[(OP2)] (Associativity) 		
for any $\lambda \in \ip(l)$, $\mu\in \ip(m)$ and $\nu\in \ip(n)$,
\[\begin{cases}
	(\la\ucr{i}\mu)\ucr{i-1+j}\nu
	=\la\ucr{i}(\mu\ucr{j}\nu),
	& 1\le i\le l, 1\le j\le m,\\
	(\la \ucr{i}\mu)\ucr{k-1+m}\nu
	=(\la\ucr{k}\nu)\ucr{i}\mu,
	& 1\le i<k\le l.
\end{cases}
\]
\end{enumerate}
\end{Def} 
If $\ip(0)=0$ (respectively, $\ic(0)=0$), the operad $\ip$ is called a \textit{reduced operad} (respectively, \textit{reduced cooperad}). 
Throughout this paper, all operads and cooperads are reduced. 
\begin{Def}\cite{LV}\label{def-Hadamrd tensor prod}
Let $\ip$ and $\ip'$ be two ns operads, the  \textit{Hadamard tensor product} $\ip\underset{\rm H}{\ot}\ip'$ is the operad defined as follows: 
\[(\ip\underset{\rm H}{\ot}\ip')(n)=\ip(n)\ot \ip'(n),\] 
and for $\mu\ot \mu'\in(\ip\underset{\rm H}{\ot}\ip')(n)$,  $\nu\ot \nu'\in(\ip\underset{\rm H}{\ot}\ip')(m)$, the partial composition is given by
$$(\mu\ot \mu')\ucr i(\nu\ot \nu')=(\mu\ucr i \nu)\ot(\mu'\ucr i \nu').$$
 Similarly, the  Hadamard tensor product of two ns cooperads $(\ic,\Delta_{\ic},\varepsilon_{\ic})$ and $(\ic',\Delta_{\ic'},\varepsilon_{\ic'})$ is the ns cooperad $\ic\underset{\rm H}{\ot}\ic'$ with $(\ic\underset{\rm H}{\ot}\ic')(n)=\ic(n)\ot\ic'(n)$ and the  coproduct  given by the composite map 
	\[\ic\underset{\rm H}{\ot}\ic' \xrightarrow{\Delta_{\ic}{\ot}\Delta_{\ic'}}  (\ic \circ \ic)\underset{{\rm H}}{\ot}(\ic' \circ \ic')\twoheadrightarrow (\ic\underset{\rm H}{\ot}\ic')\c (\ic\underset{\rm H}{\ot}\ic').  \]
\end{Def}
\begin{exm}
Let $V$ be a  vector space.
The \textit{endomorphism operad} $\End_V$ of $V$ is defined as follows:
$\End_V(n)\colon=\Hom(V^{\ot n}, V)$ for all $n\ge 0$,  
the identity $\1:=\Id_V\in \mathcal{E}nd_V(1)$ and the partial composition $-\ucr{i}-$ is given by
	\begin{align}\label{1-1-def-partial composition of map}\tag{E1.1.1}
		(f\ucr{i}g)(v_1,v_2,\cdots,v_{m+n-1})=f(v_1,\cdots,v_{i-1},g(v_{i},\cdots,v_{m+i-1}),v_{m+i},\cdots,v_{m+n-1}).
\end{align}
for  $f\in \End_V(m)$, $g\in \End_V(n)$ and $1\le i\le m$. 
\end{exm}

\subsection{Algebra over an operad and coalgebra over a cooperad}

Let $(\ip, \gamma, \eta)$ be a ns operad. 
Recall that 
a \textit{$\ip$-algebra} is a vector space $A$ equipped with a linear map
$\gamma_A \colon \ip(A)\to A$, where $\ip(A)=\bigoplus_{n} \ip(n)\otimes A^{\otimes n}$, such that the following diagrams
commute:
\[\xymatrix{
  \ip(\ip(A)) \ar[r]^-{\ip(\gamma_A)}\ar[d]_-{\gamma(A)}  & \ip(A)\ar[d]^{\gamma_A}\\
 \ip(A) \ar[r]^-{\gamma_A} & A,
} \quad\quad \quad
\xymatrix{
\I(A)\ar@{=}[dr] \ar[r]^{\eta(A)} & \ip(A) \ar[d]^{\gamma_A}\\
& A.
}\]
Equivalently, there exists a morphism $\widetilde{\gamma}_A\colon \ip \to \End_A$ of operads. 
In this case, for any $\theta\in \ip(n)$ and $n\ge 0$,  $\widetilde{\gamma}_A(\theta)\in \Hom(A^{\otimes n}, A)$
gives an $n$-ary operation of $A$, still denoted by $\theta$ for brevity. 

Let $V$ be a vector space. Recall that a \textit{free} $\ip$-algebra is the vector space $\ip(V)$ equipped with a $\ip$-algebra structure 
\[\gamma_{\ip(V)}\colon=\gamma(V)\colon \ip(\ip(V)) \to \ip(V).\]

\begin{exm}
Suppose that $\mathcal{A}\mathsf{s}$ is a reduced ns operad with $\mathcal{A}\mathsf{s}(n)=\Bbbk \mu_n$ for all $n\ge 1$ and the partial composition $-\ucr i-$ is given by
$$\mu_m\ucr i\mu_n=\mu_{m+n-1}$$
for all $1\le i\le m$ and $n \ge 1$.
Clearly, $\mu_1$ is the identity of $\mathcal{A}\mathsf{s}$.  
The operad $\mathcal{A}\mathsf{s}$ is called the \textit{non-unital associative operad},
which encodes the category of non-unital associative algebras. In fact, 
 $A$ is an $\mathcal{A}\mathsf{s}$-algebra if and only if $A$ is a non-unital associative algebra.
\end{exm}

Let $(\ic,\Delta,\varepsilon)$ be a ns cooperad.
Recall that a \textit{${\ic}$-coalgebra} is a vector space $C$ equipped with a map $\Delta_{C}\colon C\to \widehat{\ic}(C)$, where $\widehat{\ic}(C)\colon=\prod_{n}\ic(n)\otimes C^{\otimes n}$,
such that the following diagrams commute:
\[\xymatrix{
	C \ar[d]_-{\Delta_{C}} \ar[r]^-{\Delta_{C}} & \widehat{\ic}(C) \ar[d]^-{\widehat{\ic}(\Delta_{C})} \\
	\widehat{\ic}(C) \ar[r]_-{\Delta(\Id_C)} & \widehat{\ic}(\widehat{\ic}(C)),
}\qquad\qquad\qquad
\xymatrix{
	C \ar@{=}[dr] \ar[r]^-{\Delta_{C}} & \widehat{\ic}(C) \ar[d]^-{\varepsilon(C)} \\
	& \I(C).
}\]
For convenience, we denote by $[\theta|v_1, \cdots, v_n]$ an element $\theta\otimes v_1 \otimes \cdots \ot v_n\in \ic(n)\otimes C^{\ot n}$. 

Let $V$ be a vector space.  Recall that the \textit{cofree} $\ic$-coalgebra is the vector space $\ic(V)$ equipped with a $\ic$-coalgebra structure
\begin{align*}
\Delta_{\ic(V)}\colon \ic(V) \to \prod_n \ic(n)\otimes (\ic(V))^{\ot n},
\end{align*}
defined by
\[\Delta_{\ic(V)}([\theta| v_1, \cdots, v_n])=\prod [c_i|[c_{i_1}|v_1, \cdots, v_{n_{i_1}}],\cdots, [c_{i_s}|v_{n_{i_1}+\cdots+n_{i_{s-1}}+1}, \cdots, v_{n}]]  
\]
for $\theta\in \ic(n)$, $v_1,\cdots,v_n \in V$ and $\Delta(\theta)=\prod \theta_i\otimes(\theta_{i_1}\ot\cdots\ot \theta_{i_s})$ with $\theta_{i_j}\in \ic(n_{i_j})$, $j=1, \cdots, s$.

Let $(\mathcal{C},\Delta,\varepsilon)$ be a  ns cooperad. 
Denote by $\mathcal{C}^*$  the linear dual operad of the cooperad $\mathcal{C}$.
By definition, $\mathcal{C}^\ast(n)=\mathcal{C}(n)^*=\Hom(\mathcal{C}(n),\Bbbk)$, 
the element $\varepsilon\in \mathcal{C}^\ast(1)$ is just the identity of  $\mathcal{C}^*$,
and the composition map 
\[\mathcal{C}^\ast(n) \otimes \mathcal{C}^\ast(k_1) \otimes \cdots\otimes \mathcal{C}^\ast(k_n) \to \mathcal{C}^\ast(k_1+\cdots+k_n)  \]
\[(f\circ (f_1, \cdots, f_n))(\theta)=(f\otimes f_1 \otimes \cdots \otimes f_n) \Delta_{n;k_1, \cdots, k_n}(\theta) \]
for any $f\in \mathcal{C}^\ast(n)$, $f_i\in \mathcal{C}^\ast(k_i)$, $i=1, \cdots, n$, and $\theta \in \mathcal{C}(k_1+\cdots+k_n)$,
where $\Delta_{n;k_1, \cdots, k_n}$ is the projection of $\Delta(\theta)$ onto $\mathcal{C}(n) \otimes \mathcal{C}(k_1) \otimes \cdots\otimes \mathcal{C}(k_n)$.
Similarly, for a ns operad $\ip$ with each $\ip(n)$ being finite dimensional, one can obtain its linear dual cooperad $\ip^*$.

The notions of $\N$-graded vector spaces and operads can be extended to the graded framework. 
A graded $\N$-graded vector space $M$ is an $\N$-graded vector space in the category of graded vector spaces.
The component of arity $n$ is a graded vector space $\{M_p(n)\}_{p\in \Z}$ for any $n$.
Equivalently, $M$ can be viewed as a family of $\N$-graded vector spaces $\{M_p\}_{p\in \Z}$. 
Notice that a graded $\N$-graded vector space $M$ is a family indexed by two labels: the degree $p$ and the arity $n$.
The degree of an element $\mu\in M(n)$ is denoted by $|\mu|$, see \cite[Section 6.2]{LV} for more details.

\subsection{Free operads and cofree cooperads}

Let $\M$ be an $\N$-graded vector space. Recall that the \textit{free operad} generated by $\M$ is an operad $\F(\M)$ with a morphism
$\eta_\M\colon \M\to \F(\M)$ of $\N$-graded vector spaces which satisfies the following universal property: for any operad $\ip$ and any  morphism
$f\colon \M \to \ip$ of $\N$-graded vector spaces, there exists a unique operadic morphism $\widetilde{f}\colon \F(\M)\to \ip$ such that $\widetilde{f}\circ \eta_\M=f$.

Recall that the \textit{cofree cooperad} generated by $\M$ is the conilpotent cooperad $\F^c(\M)$ with a morphism $\pi\colon \F^c(\M)\to \M$ of $\N$-graded vector space  
satisfying the following universal property:
for any conilpotent cooperad $\mathcal{C}$ and any morphism $\varphi\colon \mathcal{C} \to \M$ of $\N$-graded vector space,
there exists a unique cooperad morphism $\widetilde{\varphi}\colon \mathcal{C} \to \F^c(\M)$ 
such that $\pi\circ \widetilde{\varphi}=\varphi$. 

On the explicit constructions of $\F(M)$ and $\F^c(M)$ generated by $\M$, we refer to Section 5.5 and Section 5.8.6 in \cite{LV}.

Recall that the\textit{ weight grading} on the free operad: the weight is the number of generating operations
needed in the construction of a given operation of the free operad. To be precise, the weight $w(\mu)$ of an operation 
$\mu$ of $\F(\M)$ is defined as follows: $w({\1})=0$, $w(\mu)=1$ for $\mu\in M(n)$, and more generally
\[w(\mu; \mu_1, \cdots, \mu_n)=w(\mu)+w(\mu_1)+\cdots+w(\mu_n).\]
We denote by $\F(\M)^{(r)}$ the $\N$-graded vector space of operations of weight $r$. The weight grading on the cofree cooperad $\F^c(\M)$ 
is defined similarly.

\subsection{Koszul duality of a quadratic operad}

By definition, an \textit{operadic quadratic data} $(E, R)$ is a graded $\N$-graded vector space $E$ and a graded sub-$\N$-graded vector space
$R\subset \F(E)^{(2)}$. The \textit{quadratic operad associated to} $(E, R)$ is defined to be
\[\ip(E, R)\colon=\F(E)/(R),\]
where $(R)$ is the operadic ideal of the free operad $\F(E)$ generated by $R$. 
The quadratic  cooperad $\mathcal{C}(E, R)$ associated to $(E, R)$ is defined to be 
the sub-cooperad of the cofree cooperad $\F^c(E)$ which is universal among the sub-cooperad $\mathcal{C}$ of $\F^c(E)$
such that the composite 
\[\mathcal{C} \hookrightarrow \F^c(E) \twoheadrightarrow \F^c(E)^{(2)}/R\]
is zero. 

Let $\mathcal{S}\colon=\End_{s\Bbbk}$ be the operad of endomorphisms of the $1$-dimensional vector space put in degree $1$.
The operadic suspension of an operad $\ip$ is defined to be the Hadamard product $\mathcal{S}\underset{{\rm H}}\otimes \ip$. 
Similarly, let $\mathcal{S}^c$ be the cooperad $\End_{s\Bbbk}^c$, and the suspension of a cooperad $\mathcal{C}$
is defined to be $\mathcal{S}^c\underset{{\rm H}}\otimes \mathcal{C}$.

The Koszul dual cooperad of the quadratic operad $\ip=\ip(E, R)$ is the quadratic cooperad 
\begin{align}\label{1-eq-Koszul dual cooperad}\tag{E1.4.1}
\ip^{\ac}\colon=\mathcal{C}(sE, s^2R),
\end{align}
where $sE$ is the $\N$-graded vector space $E$ whose degree is shifted by $1$. Observe that $\ip^{\ac}=\mathcal{C}(E, R)$ as an $\N$-graded vector space.
The Koszul dual operad of $\ip$ is the quadratic operad
\begin{align}\label{1-eq-Koszul dual operad}\tag{E1.4.2}
\ip^!\colon=(\mathcal{S}^c\underset{{\rm H}}\otimes \ip^{\ac})^\ast.
\end{align}

\begin{prop}[{\cite[Proposition 7.2.1]{LV}}]\label{1-prop-Koszul dual operad}
For any quadratic operad $\ip=\ip(E, R)$, generated by a reduced $\N$-graded vector space $E$ of finite dimensional in each arity, 
the Koszul dual operad $\ip^!$ admits a quadratic presentation of the form
\[\ip^!=\ip(s^{-1}\mathcal{S}^{-1}\underset{\rm H}\otimes E^\ast, R^\perp).\]	
\end{prop}

%

\subsection{Tree monomials and rewriting method}

In this subsection, we recall the rewriting method to prove the koszulity  of an operad, for details, see \cite[Chapter 8]{LV} and \cite[Chapter 3]{BD}. By definition, a \textit{rooted tree} $T$ is a finite set $\vert(T)$ (called the set of vertices) together with a parent function $\pnt_T$, where $\vert(T)= \int(T)\sqcup \leav(T ) \sqcup \{r\}$ is the disjoint union of 
the set $\int(T)$ of internal vertices,  the set $\leav(T )$ of leaves and an assigned vertex $r$ (called the root),
the parent function $\pnt_T\colon \vert(T)\setminus\{r\}\ra \vert(T)$ satisfies $|\pnt^{-1}_T(r)|=1$, $|\pnt^{-1}_T(l)|=\emptyset$ for any $l\in \leav(T)$, and for each vertex $v \in \vert(T)\setminus\{r\}$, there is a (unique) positive integer $n$ and vertices
	$v_0=v,v_1,\cdots,v_n=r$, such that $v_i=\pnt_T(v_{i-1})$
	for all $1\leq i\leq n$.

A \textit{planar rooted tree} $T$ is a rooted tree  such that for each $v\in \vert(T)$, there is a total order on the preimage $\pnt^{-1}_T(v)$.
In this paper, a planar rooted trees is drawn as a tree in the plane such that the order on $\pnt^{-1}(v)$ is determined by ordering the corresponding edges left-to-right and the order of leaves of $T$ is also sorted from the left to the right. 
For example, let $T$ be a planar rooted tree $T$ with $\int(T)=\{v_1,v_2,v_3\}$, $\leav(T )=\{l_1,l_2,l_3,l_4\}$ and $\pnt^{-1}(r)=\{v_1\}$, $\pnt^{-1}(v_1)=\{v_2<l_3<v_3\}$,  $\pnt^{-1}(v_2)=\{l_1<l_2\}$, $\pnt^{-1}(v_3)=\{l_4\}$, 
and the total order on the leaves is $\{l_1<l_2<l_3<l_4\}$.
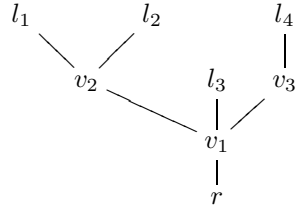
\begin{equation*}\begin{array}{c}
		$\xymatrix@M=3pt@C=10pt@R=10pt{
			l_1	\ar@{-}[dr] & & l_2\ar@{-}[dl] & & l_4 \ar@{-}[d] \\
			& v_2\ar@{-}[drr]  & &l_3 \ar@{-}[d] &  v_3\ar@{-}[dl]\\
			& &  &v_1 \ar@{-}[d] & \\
			&&&r&
		}$\end{array} 
\end{equation*}

Let $\mx=\{\mx(n)\}_{n\ge 0}$ be a collection of finite sets $\mx(n)$, which is called an \textit{operation alphabet}. 
A \textit{tree monomial} in $\mx$ is a pair $(T, \lb)$, where $T$ is a planar rooted tree and $\lb\colon \int(T)\to \mx$ is a map satisfying
$\lb(v)\in \mx(|\pnt^{-1}(v)|)$ for any $v\in \int(T)$. 
In fact, a tree monomial gives a labeling of all internal vertices of $T$ by elements of $\mx$.

Let $\mx=\{\mx(n)\}_{n\ge 0}$ be an operation alphabet with a total order.  
Suppose that $E$ is an $\N$-Mod  generated by $\mx$, that is, each $E(n)$ is the $\Bbbk$-vector space spanned by $\mx(n)$ for $n\geq 0$.
Then the set of all  tree monomials in $\mx$ form a $\Bbbk$-linear basis of the free ns operad $\mathcal{F}(E)$.
For each tree monomial $(T, \lb)\in \mathcal{F}(E)$, we record the labels of internal vertices of
the path from the root  to each leaves  and arrange it according to the order of leaves.
Then one get a sequence of words in the alphabet $\mx$, denoted by {\rm Path}$(T, \lb)$,
which is called \textit{path sequence} of the tree monomial  $(T, \lb)$.
For example, we consider the following three tree monomials
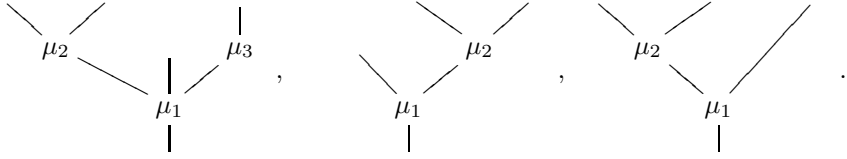
\begin{equation*}\begin{array}{c}
		$\xymatrix@M=3pt@C=10pt@R=10pt{
			\ar@{-}[dr] & & \ar@{-}[dl] & &  \ar@{-}[d] \\
			& \mu_2\ar@{-}[drr]  & & \ar@{-}[d] &  \mu_3\ar@{-}[dl]\\
			& &  &\mu_1 \ar@{-}[d]& \\
			& & &&
		}$\end{array}, \begin{array}{c}
		$\xymatrix@M=3pt@C=10pt@R=10pt{
			& & \ar@{-}[dr] & &  \ar@{-}[dl] \\
			& \ar@{-}[dr]  & & \mu_2\ar@{-}[dl] & \\
			& & \mu_1 \ar@{-}[d]& & \\
			& & &&
		}$\end{array},  \begin{array}{c}
		$\xymatrix@M=3pt@C=10pt@R=10pt{
			\ar@{-}[dr] & & \ar@{-}[dl] & &  \ar@{-}[ddll] \\
			& \mu_2\ar@{-}[dr]  & &  & \\
			& & \mu_1 \ar@{-}[d] & & \\
			&&&&
		}$\end{array}.
\end{equation*}
Then the corresponding path sequences are $(\mu_1\mu_2,\mu_1\mu_2,\mu_1,\mu_1\mu_3)$, $(\mu_1,\mu_1\mu_2,\mu_1\mu_2)$, $(\mu_1\mu_2,\mu_1\mu_2,\mu_1)$,
respectively.

The \textit{path-lexicographic order} on the set of all tree monomials in $\mx$ is defined as follows: for two tree monomials $(T, \lb)$ and $(T', \lb')$,
\begin{itemize}
	\item  if $|\leav(T)|<|\leav(T')|$, we put $(T, \lb) < (T', \lb')$;
	\item  if $|\leav(T)|=|\leav(T')|$, we compare the path sequences {\rm Path}$(T, \lb)$ and {\rm Path}$(T', \lb')$ word by word, using the graded lexicographic order of $\mx^\ast$, which is the set of all words in the alphabet $\mx$.	 
\end{itemize}

Let $(E, R)$ be a quadratic data and $\ip(E, R)=\mathcal{F}(E)/(R)$ the associated quadratic operad, 
where $E$ is an $\N$-graded vector space spanned by an operation alphabet $\mx=\{\mx(n)\}_{n\ge 0}$ with a total order, and
$R\subset \mathcal{F}(E)^{(2)}$ is the set of relators. 
Observe that the space $R$ of relations admits a $\Bbbk$-basis $\mathcal{B}_R$ in which the element is of the form
\begin{align}\label{1-eq-the form of Grobner basis}\tag{E1.5.1}
(\gamma_{i,a,j})=\mu_i\ucr{a}\mu_j-\sum\lambda_{k,b,l}^{i,a,j}~\mu_k\ucr{b}\mu_l,\quad 0\neq\lambda_{k,b,l}^s\in\Bbbk,
\end{align}
such that the sum runs over the terms satisfying $\mu_k\ucr{b}\mu_l< \mu_i\ucr{a}\mu_j$ and $\mu_{q}\in \bigcup_{n\geq 0}\mx(n)$ for any $q\geq 0$. Here the tree monomial $\mu_i\ucr{a}\mu_j$ is called the \textit{leading term} of the relator $(\gamma_{i,a,j})$.
One can always assume that  the leading terms of the set of relators are all distinct and there is no leading term of any other relator in the sum of the right  side. 
Note that the relator $(\gamma_{i,a,j})$ give rise to a rewriting rule in the operad $\ip(E, R)$ by
\[\mu_i\ucr{a}\mu_j\mapsto\sum\lambda_{k,b,l}^{i,a,j}~\mu_k\ucr{b}\mu_l.\]
A tree monomial $(T, \lb) \in \mathcal{F}(E)^{(3)}$ is said to be \textit{critical} if the two sub-trees with $2$ internal vertices are leading terms of some relators. Clearly, there are at least two ways to rewrite a critical monomial. A critical monomial is said to be \textit{confluent} if all these rewriting rules lead to the same element until no rewriting rule is applicable any more. The following theorem plays a crucial role to prove a ns operad is Koszul.

\begin{thm}[{\cite[Theorem 8.1.1]{LV}}]\label{1.3-thm-rewriting method to Koszul}
	Let $\ip(E,R)$ be a reduced quadratic ns operad, where $E$ is an $\N$-graded vector space generated by an operation alphabet $\mx$. 
	If the operation alphabet $\mx$ admits a total order for which there exists a suitable order on planar trees 
	such that every critical monomial is confluent, then $\ip(E, R)$ is Koszul.
\end{thm}
In this case, the basis $\mathcal{B}_R$ of the relation space $R$ is called a \textit{Gr\"{o}bner basis} of the operadic ideal $(R)$.
The  following result is useful to find a $\Bbbk$-basis of a ns operad $\mathcal{F}(E)/(R)$.

\begin{prop}\cite[Proposition 3.4.3.4.]{BD} \label{1.3-prop-basis of koszul operad}
	Let $\mathcal{I}$ be an operadic ideal of free operad $\mathcal{F}(E)$. Then $\mathcal{B} \subset \mathcal{I}$ is a Gr\"{o}bner basis if and only if the cosets of tree monomials that are reduced with respect to $\mathcal{B}$ form a $\Bbbk$-basis of the quotient operad $\mathcal{F}(E)/\mathcal{I}$.
\end{prop}

\subsection{Generalized derivation}
Let $A$ be an associative algebra. 
Recall that a linear map $\delta\colon A\to A$ is called a \textit{generalized derivation} of $A$ if 
there exists a derivation $h$ of $A$ such that 
\[\delta(ab)=\delta(a)b+ah(b)
\]
for all $a, b\in A$, see \cite{Br}. In this case, $\delta$ is also called a \textit{generalized $h$-derivation}.

Here we give some examples of generalized derivations.

\begin{exm}[\cite{Br}]
Let $A$ be an associative algebra. 
Fixing two elements $a, b\in A$, we consider the linear map
\[\delta_{a, b}\colon A\to A, \quad \delta_{a, b}(x):=ax+xb,\]
for all $x\in A$. 
Clearly, $\delta_{a, b}$ is a generalized $\operatorname{ad}_b$-derivation.
In this case, $\delta_{a, b}$ is called a \textit{generalized inner derivation}.
\end{exm}

\begin{exm}
Let $V$ be a vector space, and $\overline{T}(V)=\bigotimes_{n\geq 1}V^{\otimes n}$ be the reduced tensor algebra with the concatenation product. 
Suppose $d$ and $\delta$ are linear maps from $V$ to $\overline{T}(V)$. 
Define the linear maps $\overline{d},\overline{\delta}\colon \overline{T}(V) \rightarrow \overline{T}(V)$ given by: 
\begin{align*}
	\overline{d} (v_1\otimes v_2\cdots\ot v_n)&=\sum_{i=1}^n v_1\ot \cdots \ot d(v_i)\ot\cdots\ot v_n,	\\
	\overline{\delta} (v_1\otimes v_2\cdots\ot v_n)&=\delta(v_1)\ot v_2\ot \cdots \ot v_n+\sum_{i=2}^n v_1\ot \cdots \ot d(v_i)\ot\cdots\ot v_n.	
\end{align*}
It is easy to verify that $\overline{d}$ is a derivation of $\overline{T}(V)$ and $\overline{\delta}$ is a generalized $\overline{d}$-derivation. In fact, any  generalized derivation on  $\overline{T}(V)$ arises in this way.	
\end{exm}

Let $A$ be a unital associative algebra with the identity $1_A$.
Suppose that $h$ is a derivation of $A$ and $\delta$ is a generalized $h$-derivation. Then, for any $a\in A$, we have
\[\delta(a)=\delta(1_A\cdot a)=\delta(1_A)\cdot a+h(a).\]
Therefore, $\delta$ is uniquely determined by the element $\delta(1_A)\in A$.


Let $A$ be a non-unital associative algebra, and $\delta$ be a generalized $h$-derivation of $A$. 
Then the triple $(A, h, \delta)$ is called an \textit{AsGDer triple}.
Let $(A, h_A, \delta_A)$ and $(B, h_B, \delta_B)$ be AsGDer triples. 
A \textit{morphism} of AsGDer triples is defined to be   
a homomorphism $f\colon A \to B$ of associative algebras satisfying
\[h_B\circ f=f\circ h_A,\quad \delta_B\circ f=f\circ\delta_A.\]
We denote by $\mathbf{AsGDer}$ the category of AsGDer triples.

\section{Operads of associative algebra with generalized derivations }

In this section, we construct the operad $\Agd$ that encodes AsGDer triples and calculate the Koszul dual operad of $\Agd$. 

\subsection{The AsGDer operad}

\begin{Def}\label{2-def-operad Agd}
Let $(E, R)$ be a quadratic data, where $E$ is an $\N$-graded vector space with $E(1)=\Bbbk h\oplus \Bbbk \delta$, $E(2)=\Bbbk \mu$ and $E(n)=0$ for all $n\neq 1, 2$, and $R$ is a sub-$\N$-graded vector space of $\F(E)$ spanned by 
\begin{align}\label{2-eq-def of relation on operad Ad}\tag{E2.1.1}
	\mu\ucr 1\mu- \mu\ucr 2\mu, \quad 
		h\ucr 1 \mu -\mu\ucr 1 h-\mu\ucr 2 h\quad {\rm and} \quad 
		\delta\ucr 1 \mu - \mu\ucr 1 \delta-\mu\ucr 2 h.	
\end{align}
Then the quadratic operad associated to the quadratic data $(E, R)$ is called \textit{AsGDer operad}, denoted by $\Agd$.
\end{Def}

Let $A$ be an $\Agd$-algebra given by the operad morphism $\gamma\colon \Agd\to \End_A$. 
For convenience, we still denote the operation $\gamma_1(h)$, $\gamma_1(\delta)$ and $\gamma_2(\mu)$ by $h$, $\delta$ and $\mu$, respectively.
By the definition of $\Agd$, $\mu$ is an associative multiplication of $A$, $h$ is a derivation of $A$,
and $\delta$ is a generalized $h$-derivation of $A$.
Therefore, $(A, h, \delta)$ is an AsGDer triple.
Conversely, it is easily seen that an AsGDer triple $(A, h, \delta)$ is also an algebra over the operad $\Agd$.

\begin{thm}\label{2-thm-Adg is Koszul}
	The ns operad $\Agd$ is Koszul.
\end{thm}

\begin{proof}
By Definition \ref{2-def-operad Agd}, we consider the ordering of generators of $\As$ given by
$\mu<h< \delta$, and the leading terms of relations $R$ are
$\mu\ucr 1\mu$, $h\ucr 1\mu$ and $\delta\ucr 1\mu$, respectively. They give rise to three critical monomials:
\[\mu\ucr 1(\mu\ucr 1\mu), \quad h\ucr1(\mu\ucr 1\mu),\quad  \delta\ucr 1 (\mu\ucr 1\mu).	\]
To prove $\Agd$ is Koszul, we only need to check that these critical monomials are confluent, it is well-known that $\mu\ucr 1(\mu\ucr 1\mu)$ is confluent, for $\delta\ucr 1 (\mu\ucr 1\mu)=(\delta\ucr 1 \mu)\ucr 1\mu$, we have 
\begin{align*}
	\delta\ucr 1 (\underline{\mu\ucr 1\mu}) &\mapsto \delta\ucr 1 (\mu\ucr 2\mu)= (\underline{\delta\ucr 1 \mu})\ucr 2\mu \mapsto 
	(\mu\ucr1 \delta+\mu\ucr2 h)\ucr 2\mu=(\mu\ucr1 \delta)\ucr 2\mu+
	\mu\ucr2 (\underline{h\ucr 1\mu}) \\
	&\mapsto(\mu\ucr1 \delta)\ucr 2\mu+\mu\ucr2 (\mu\ucr1 h)+\mu\ucr2(\mu\ucr2 h)=(\mu\ucr1 \delta)\ucr 2\mu+(\mu\ucr2 \mu)\ucr2 h+(\mu\ucr2\mu)\ucr3 h,\\
	(\underline{\delta\ucr 1 \mu})\ucr 1\mu&\mapsto (\mu\ucr 1\delta)\ucr 2\mu+(\mu\ucr 2 h)\ucr 2\mu=(\mu\ucr 1\delta)\ucr 2\mu+\mu\ucr 2(\underline{h\ucr 1\mu} )\\
	&\mapsto (\mu\ucr 1\delta)\ucr 2\mu+ \mu\ucr 2(\mu\ucr1 h )+\mu\ucr2(\mu\ucr 2 h)=(\mu\ucr 1\delta)\ucr 2\mu+ (\mu\ucr 2\mu)\ucr2 h +(\mu\ucr2\mu)\ucr 3 h.
\end{align*}
Hence the critical monomial $\delta\ucr 1 (\mu\ucr 1\mu)$ is confluent.
Similarly, one can show that $h\ucr 1(\mu\ucr 1\mu)$ is confluent. By Theorem \ref{1.3-thm-rewriting method to Koszul}, we know that the operad $\Agd$ is Koszul.
\end{proof}

\subsection{The Koszul dual operad of $\Agd$}
\begin{thm}\label{2-thm-Koszul dual of Agd}
The Koszul dual operad $\Agd^!$ of $\Agd$ is the operad generated by two unary operations  $\bar h, \bar \delta$ of degree $-1$ and one binary operation $\bar\mu$ of degree $0$, which satisfy the following relations:
\begin{align}\label{2-eq-generate relation of dual operad}\tag{E2.2.1}
	\begin{aligned}
		&\bar\mu\ucr 1\bar\mu=\bar\mu\ucr 2\bar\mu,           &&~~\bar\mu\ucr 1 \bar h=\bar h\ucr 1\bar\mu,  &&~~\bar\mu\ucr 1\bar\delta=\bar\delta\ucr 1\bar\mu,\\
		&\bar\mu\ucr 2 \bar h=\bar\delta \ucr 1\bar\mu+\bar h\ucr 1\bar\mu, &&~~\bar\mu\ucr 2\bar\delta=0,        &&~~\bar h\ucr 1\bar h=0,\\
		&\bar h\ucr 1\bar\delta=0,                       &&~~\bar\delta\ucr 1\bar\delta=0,       &&~~\bar\delta\ucr 1\bar h=0.
	\end{aligned} 			
\end{align}
\end{thm}
\begin{proof}
By Definition \ref{2-def-operad Agd} and Equation \eqref{1-eq-Koszul dual cooperad}, we have that the Koszul dual cooperad $\Agd^{\ac}=\mathcal{C}(sE,s^2R)$ of $\Agd$ is cogenerated by $sh, s\delta$ and $s\mu$ with the corelations $s^2R$. 
Therefore, the linear dual operad $(\Agd^{\ac})^*=\ip(s^{-1}E^*,(s^2R)^{\bot})$ is a quadratic operad generated by $s^{-1}h^*, s^{-1}\delta^* , s^{-1}\mu^*$ and satisfies the relations $(s^2R)^{\bot}$ spanned by the following elements:
\begin{equation*}
	\begin{aligned}
		&s^{-1}\mu^*\ucr 1s^{-1}\mu^*+s^{-1}\mu^*\ucr 2s^{-1}\mu^*, &&~~s^{-1}\mu^*\ucr 1s^{-1}h^*+s^{-1}h^*\ucr 1s^{-1}\mu^*,\\  &s^{-1}\mu^*\ucr 1s^{-1}\d^*+s^{-1}\d^*\ucr 1s^{-1}\mu^*,
		&&s^{-1}\mu^*\ucr 2s^{-1}h^*+s^{-1}\d^*\ucr 1s^{-1}\mu^*-s^{-1}h^*\ucr 1s^{-1}\mu^*,\\
		&s^{-1}\mu^*\ucr 2s^{-1}\d^*,  &&~~s^{-1}h^*\ucr 1s^{-1}h^*,\\
		&s^{-1}h^*\ucr 1s^{-1}\d^*,  &&~~s^{-1}\d^*\ucr 1s^{-1}\d^*,\\ &s^{-1}\d^*\ucr 1s^{-1}h^*.&&
	\end{aligned} 			
\end{equation*}
By Proposition \ref{1-prop-Koszul dual operad}, we have $\Agd^!=(S^{c})^*\underset{\rm H}{\ot}(\Agd^{\ac})^*$. 
Denote 
\[\bar\mu\colon=v_2{\ot}s^{-1}\mu^*, \quad \bar h\colon=v_1{\ot}s^{-1}h^*, \quad {\rm and}\quad \bar\delta\colon=v_1{\ot}s^{-1}\delta^*.\]
Clearly, $|\bar\mu|=0$ and $|\bar h|=|\bar\delta|=-1$. By direct calculation, we know that $\Agd^!$ can be generated by $\bar\mu$, $\bar h$, and $\bar\delta$ with the relations \eqref{2-eq-generate relation of dual operad}.
\end{proof}

The operad $\Agd^{!}$ is also a quadratic operad that is generated by $\N$-Mod 
\[\bar E=(0,\Bbbk \bar h\oplus \Bbbk \bar \delta, \Bbbk \bar\mu,0,\cdots)\]
with the relation $R^\perp$ spanned by the elements in \eqref{2-eq-generate relation of dual operad}.
Define  the ordering of generators of $\Agd^{!}$ by
$\bar\delta< \bar h< \bar \mu$,  one can similarly show that all of the corresponding critical tree monomials are confluent.
It follows from Proposition \ref{1.3-prop-basis of koszul operad} that the tree monomials of $\mathcal{F}(\bar E)$ that are reduced with respect to  the identities in \eqref{2-eq-generate relation of dual operad} form a basis of the quotient operad $\mathcal{F}(\bar E)/(R^\perp)$.

\begin{prop}\label{2-prop-def of koszul dual operad of Agd}
Let $\Agd^!$ be the Koszul dual operad of the AsGDer operad $\Agd$. Then 
\[\As^!(n)= \Bbbk\bar\mu_n \oplus \Bbbk \bar h_n\oplus\Bbbk \bar\delta_n\]
with $|\bar\mu_{n}|=0$ and $|\bar h_n|=|\bar\delta_n|=-1$ for $n\ge 1$, $\bar\mu_1\in \Agd(1)$ is the identity, and the partial composition is given by
\begin{equation}\label{2-eq-partial comp of koszul dual of Agd}\tag{E2.2.2}
\begin{array}{lll}
	\bar\mu_n\ucr i\bar\mu_m=\bar\mu_{m+n-1}, & \bar h_n\ucr i\bar\mu_m=\bar h_{m+n-1}, & \bar\delta_n\ucr i\bar\mu_m=\bar\delta_{m+n-1},\\
	\bar\mu_n\ucr i \bar h_m=\begin{cases}
		\bar h_{m+n-1}, &i=1,\\
		\bar\delta_{m+n-1}+\bar h_{m+n-1}, & 2\leq i\leq n,
	\end{cases} & \bar h_n\ucr i\bar h_m=0, & \bar\delta_n\ucr i\bar h_m=0,\\
\bar\mu_n\ucr i\bar\delta_m=\begin{cases}
	\bar\delta_{m+n-1}, &i=1\\
	0, &2\leq i\leq n,
\end{cases} & \bar h_n\ucr i\bar\delta_m=0, & \bar\delta_n\ucr i\bar\delta_m=0.
\end{array}
\end{equation}
In particular, $\bar h_1=\bar h$, $\bar\delta_1=\bar\delta$ and $\bar \mu_2=\bar \mu$ as in Theorem \ref{2-thm-Koszul dual of Agd}.
\end{prop}  
\begin{proof}
For each $n\geq 1$, there are only three monomials   
\[\bar\mu_n:=(\cdots(\bar\mu\ucr 2\bar\mu)\ucr 3\bar\mu)\ucr 4\cdots)\ucr {n-1}\bar\mu, \quad \bar h_n:=\bar h\ucr 1\bar\mu_n, \quad \bar\delta_n:=\bar\delta\ucr 1 \bar\mu_n.\] 
in $\Agd(n)$ which are  reduced with respect to relations $\eqref{2-eq-generate relation of dual operad}$.
It follows that $\Agd^!(n)$ has a basis $\{\bar\mu_n, \bar h_n, \bar\delta_n\}$ for all $n\ge 1$. 
By direct computation, we get the partial composition as in \eqref{2-eq-partial comp of koszul dual of Agd}.
\end{proof}

\subsection{Free AsGDer triple}

In this part, we study the free object in the category $\mathbf{AsGDer}$.
Let $V$ be a vector space over $\Bbbk$. We denote $\dbar{V}=\Bbbk\langle x,y\rangle \ot V$, 
where  $\Bbbk\langle x, y\rangle$ is the free associative algebra generated by two variables $x$ and $y$.
For convenience, we denote by $gv$ the homogeneous element $g\otimes v\in \Bbbk\langle x, y\rangle \ot V$. 

Let $\overline{T}(\dbar{V})$ be the reduced tensor algebra.
Define the linear maps $h,\delta: \overline{T}(\dbar{V})\to \overline{T}(\dbar{V})$ as follows:
\begin{align*}
	h(g_1v_1\ot g_2v_2\cdots\ot g_nv_n)&=\sum_{i=1}^ng_1v_1\ot \cdots\ot (xg_i) v_i\otimes\cdots \ot g_nv_n,\\
	\delta(g_1v_1\ot g_2v_2\cdots\ot g_nv_n)&=(yg_1)v_1\ot g_2v_2\cdots\ot g_nv_n+\sum_{i=2}^ng_1v_1\ot \cdots\ot (xg_i)v_i\otimes\cdots \ot g_nv_n.	
\end{align*}
for all $g_i\in \Bbbk\langle x,y\rangle $ and $v_i\in V$. By an easy calculation, we have that $(\overline{T}(\Bbbk\langle x,y\rangle \ot V),h,\delta)$ is an AsGDer triple.

\begin{prop}\label{2-prop-free asgder}
Retain the above notation. Let $V$ be a vector space over $\Bbbk$. 
Then $(\overline{T}(\dbar{V}),h,\delta)$ is a free object in the category  $\mathbf{AsGDer}$. 
To be precise, for any AsGDer $(B, h_B,\delta_B)$ and any linear map $\varphi\colon V\to B$, there exists a unique morphism 
$\overline{\varphi}\colon  \overline{T}(\dbar{V}) \to B$ of AsGDers such that the following diagram commutes:
\[\xymatrix{
V \ar[r]^-{i}\ar[dr]_\varphi & \overline{T}(\dbar{V})\ar@{-->}[d]^{\bar \varphi}\\
& B}\]
where the linear map $i$ is defined by $i(v)=1v=1\otimes v\in \dbar{V}\subset \overline{T}(\dbar{V})$ for $v\in V$.	
\end{prop}
\begin{proof}
We define a linear map $\varphi':\Bbbk\langle x,y\rangle \ot V\ra B$ as follows: for any $g=x^{i_1}y^{j_1}\cdots x^{i_s}y^{j_s} \in \Bbbk\langle x,y\rangle$ $(i_k, j_k\ge 0)$, and $v\in V$,
\[\varphi'(gv):=(h_B^{i_1}\circ \delta_B^{j_1}\circ \cdots\circ  h_B^{i_s}\circ\delta_B^{j_s}\circ \varphi )(v),\]
where $h_B^{k}$ is the composition of $k$ copies of $h_B$ for $k\ge 0$. Clearly, $\varphi'\circ i=\varphi$.
	
Since $\overline{T}(\dbar{V})$ is a free associative algebra, $\varphi'$ induces an associative algebra morphism $\overline{\varphi}:\overline{T}(\dbar{V})\to B$. Clearly, $\overline{\varphi}\circ i=\varphi$.   
It is easy to verify that $h_B\circ\overline{\varphi}=\overline{\varphi}\circ h$ and $\delta_B\circ\overline{\varphi}=\overline{\varphi}\circ\delta$,
and therefore $\overline{\varphi}$ is a morphism of AsGDer triples. 
	
It remains to  show that $\overline{\varphi}$ is unique.
Assume that there is another morphism $\widehat{\varphi}\colon \overline{T}(\dbar{V})\to B$ of  AsGDer triples such that $\widehat{\varphi}\circ i=\varphi$, then we have
\[\widehat{\varphi}(1v)=\varphi(v)=\overline{\varphi}(1v),\quad h_B\c\widehat{\varphi}=\widehat{\varphi}\circ h,
\quad \delta_B\circ\widehat{\varphi}=\widehat{\varphi}\circ\delta.\]
It follows that 
\[\widehat{\varphi}((xg)v)=\widehat{\varphi}\c h(gv)=h_B\c\widehat{\varphi}(gv) \quad \widehat{\varphi}(ygv)=\widehat{\varphi}\c\d(gv)=\delta_B\c\widehat{\varphi}(gv),\]
for any $g\in\Bbbk\langle x,y\rangle$ and $v\in V$. Therefore, we have
\[\widehat{\varphi}(gv)=h_B^{i_1}~\delta_B^{j_1}~\cdots~ h_B^{i_s}~\delta_B^{j_s}\varphi(v)=\overline{\varphi}(gv)\]
for $g=x^{i_1}y^{j_1}\cdots x^{i_s}y^{j_s} \in \Bbbk\langle x,y\rangle$, and $v\in V$. 
Therefore, we have $\widehat{\varphi}=\overline{\varphi}$ since $\overline{T}(\dbar{V})$ is a free associative algebra.
\end{proof}

\section{homotopy AsGDer Triples}

In this section, we consider $\As_{\infty}$-algebras, that is, homotopy AsGDer triples. 
For this, we first recall some notation on cooperad.
Let $M$, $N$ and $N'$ be $\mathbb{N}$-modules. The right linear analog $M\c(N,N')$ of the composition product is an $\mathbb{N}$-module defined by:
\[
[M\c(N,N')](n):=\bigoplus_k M(k)\otimes \left(\bigoplus_{i_1+\cdots+i_k=n}\bigoplus_{j=1}^k N(i_1)\otimes\cdots \otimes N'(i_j)\otimes\cdots \otimes N(i_k)\right).
\]
For $\mathbb{N}$-module morphisms $f : M \ra M'$ and $g : N \ra N'$ ,  the infinitesimal composite 
\[f\c' g:M\c N\ra M'\c(N,N')\]
is  defined as
\[f\c' g\colon=\sum_{i}f\ot (\Id_N^{\ot i-1}\ot g\ot \Id_N^{\ot n-i} ).\]

Let $(\ic,\Delta,\varepsilon)$ be a ns cooperad. For convenience, we denote $\ic \ucr{(1)} \ic\colon = \ic\circ (\I, \ic)$ and 
\[\alpha\bc{i}\beta\colon=\alpha\otimes(\1^{\ot i-1}\ot \beta\ot \1^{\ot n-i}) \in \ic\ucr{(1)}\ic\]
for all $\alpha\in \ic(n)$ and $\beta\in \ic(m)$. The infinitesimal decomposition map $\Delta_{(1)} :\ic\to \ic\ucr{(1)}\ic$ is defined by
the composite map
\[ \ic \xrightarrow{\Delta} \ic \circ \ic \xrightarrow{\Id_{\ic}\c'\Id_{\ic}}\ic\c(\ic,\ic) \xrightarrow{\Id_{\ic}\c(\varepsilon,\Id_{\ic})}\ic\c(\I,\ic)= \ic \ucr{(1)} \ic,\]
see \cite[Section 6.1]{LV}.
For convenience, we denote $\Delta_{(1)}(\theta)=\sum_{i} \theta_{(1)}\bc{i}\theta_{(2)}$ for all $\theta\in \ic(n)$.
Note that $\art(\theta_{(1)})+\art(\theta_{(2)})=n+1$.

\subsection{Koszul dual cooperad $\Agd^{\ac}$}

Let $\mathcal{S}^{-c}$ be the linear dual cooperad of the endomorphism operad $\End_{\Bbbk s}$. To be precise, $\mathcal{S}^{-c}(n)=\Bbbk \alpha_n$ with $|\alpha_n|=n-1$, 
and the infinitesimal decomposition map is defined by:
\[\Delta_{(1)} \bigl(\alpha_{n}\bigr)= \sum_{p+q+r=n}(-1)^{r(q-1)}   \alpha_{p+r+1}\bc{p+1}\alpha_q.\]		
We consider the ns cooperad $\mathcal{S}^{ c}$ given by  
$\mathcal{S}^{c}(n)=\Bbbk\beta_n$ with $|\beta_n|=1-n$ and the infinitesimal decomposition map 
\[\Delta_{(1)} \bigl(\beta_{n}\bigr)= \sum_{p+q+r=n}(-1)^{p(q-1)}  \beta_{p+r+1}\bc{p+1}\beta_q.\]		
Clearly $\Agd^!=(\mathcal{S}^{c}\underset{\rm H}{\ot}\Agd^{\ac})^*$ and $(\Agd^!)^*=\mathcal{S}^{c}\underset{\rm H}{\ot}\Agd^{\ac}$.  
Observe that there are isomorphisms of cooperads
\[\mathcal{S}^{-c}\underset{\rm H}{\ot}\mathcal{S}^c\cong \ias^* \quad {\rm and} \quad \ias^*\underset{\rm H}{\ot} \ic\cong \ic\] 
for any ns cooperad $\ic$, and therefore we have the isomorphism 
\begin{align}\label{3.1-eq-Koszul dual cooperad and dual of Koszul dual operad}\tag{E3.1.1}
\Agd^{\ac}\cong \ias^*\underset{\rm H}{\ot}\Agd^{\ac}\cong \mathcal{S}^{- c}\underset{\rm H}{\ot}\mathcal{S}^c\underset{\rm H}{\ot}\Agd^{\ac}\cong \mathcal{S}^{-c}\underset{\rm H}{\ot}(\As^!)^*
\end{align}
of ns cooperads.

By Proposition \ref{2-prop-def of koszul dual operad of Agd},  
we know that the cooperad $(\As^!)^*$ is given by 
\[(\As^!)^*(n)= \Bbbk\bar{\mu}^*_n \oplus \Bbbk \bar{h}^*_n\oplus\Bbbk \bar{\delta}^*_n,\quad ({\rm for}~n\geq 1)
\]
and the infinitesimal decomposition map 
\begin{align*}
	\Delta_{(1)}(\bar{\mu}^*_n)&=\sum_{p+q+r=n}\bar{\mu}^*_{p+r+1}\bc {p+1}\bar{\mu}^*_{q},\\
	\Delta_{(1)}(\bar{h}^*_n)&=\sum_{p+q+r=n}
	\big(\bar{\mu}^*_{p+r+1}\bc {p+1}\bar{h}^*_{q}+\bar{h}^*_{p+r+1}\bc {p+1}\bar{\mu}^*_{q}\big),\\	
	\Delta_{(1)}(\bar{\delta}^*_n)&=\sum_{p+q+r=n}\bar{\delta}^*_{p+r+1}\bc {p+1}\bar{\mu}^*_{q}+
	\sum_{q+r=n}\bar{\mu}^*_{r+1}\bc {1}\bar{\delta}^*_{q}+\sum_{p+q+r=n\atop p\geq1}
	\bar{\mu}^*_{p+r+1}\bc {p+1}\bar{h}^*_{q}.
\end{align*}
It follows that the cooperad $S^{-c}\underset{\rm H}{\ot}(\Agd^!)^*$ is given by  
\begin{align*}
	(S^{-c}\underset{\rm H}{\ot}(\Agd^!)^*)(n)&= \Bbbk (\alpha_n\ot\bar{\mu}^*_n) \oplus \Bbbk(\alpha_n\ot \bar{h}^*_n) \oplus\Bbbk(\alpha_n\ot\bar{\d}^*_n) \quad ({\rm for} ~n\geq 1),	
\end{align*}
and the infinitesimal decomposition map 
\begin{align*}
	\Delta_{(1)}( \alpha_n\ot\bar{\mu}^*_n)&=\sum_{p+q+r=n}(-1)^{r(q-1)}(\alpha_{p+r+1}\ot\bar{\mu}^*_{p+r+1})\bc {p+1}(\alpha_{q}\ot\bar{\mu}^*_{q}),\\
	\Delta_{(1)}(\alpha_n\ot \bar{h}^*_n)&=\sum_{p+q+r=n}
	(-1)^{r(q-1)}(\alpha_{p+r+1}\ot\bar{\mu}^*_{p+r+1})\bc {p+1}(\alpha_{q}\ot \bar{h}^*_{q})\\
	&+\sum_{p+q+r=n}(-1)^{(r+1)(q-1)}(\alpha_{p+r+1}\ot \bar{h}^*_{p+r+1})\bc {p+1}(\alpha_{q}\ot\bar{\mu}^*_{q}),\\	
	\Delta_{(1)}(\alpha_n\ot \bar{\delta}^*_n)&=\sum_{p+q+r=n}(-1)^{(r+1)(q-1)}(\alpha_{p+r+1}\ot \bar{\delta}^*_{p+r+1})\bc {p+1}(\alpha_{q}\ot\bar{\mu}^*_{q})\\
	&+\sum_{q+r=n}(-1)^{r(q-1)}(\alpha_{r+1}\ot\bar{\mu}^*_{r+1})\bc {1}(\alpha_{q}\ot \bar{\delta}^*_{q})\\
	&+\sum_{p+q+r=n\atop p\geq1}(-1)^{r(q-1)}(\alpha_{p+r+1}\ot\bar{\mu}^*_{p+r+1})\bc {p+1}(\alpha_{q}\ot \bar{h}^*_{q}).
\end{align*}
By the isomorphism \eqref{3.1-eq-Koszul dual cooperad and dual of Koszul dual operad} of cooperads, we denote by ${\mu}^c_n$, $h^c_n$ and  $\d^c_n$  the isomorphic image of $\alpha_n\ot\bar{\mu}^*_n$
$\alpha_n\ot \bar{h}^*_n$ and $\alpha_n\ot \bar{\delta}^*_n$, respectively. 
Then we obtain the explicit description of the Koszul dual cooperad of $\Agd$.

\begin{prop}\label{3-prop-def of Koszul dual cooperad }
Let $\Agd^{\ac}$ be the Koszul dual cooperad of $\Agd$. Then
\[
\Agd^{\ac}(n)= \Bbbk \mu^c_n\oplus \Bbbk h^c_n \oplus \Bbbk  \d^c_n \quad ({\rm for} ~ n\geq 1), 	
\]
with $|\mu^c_n|=n-1$, $|h^c_n|=|\d^c_n|=n$, the counit $\varepsilon$ is given by $\varepsilon(\mu^c_1)=\1\in \I$,
and the infinitesimal decomposition is given by
\begin{align*}
	\Delta_{(1)}( \mu^c_n)=& \sum_{p+q+r=n}(-1)^{r(q-1)}\mu^c_{p+r+1}\bc {p+1}\mu^c_{q},\\
	\Delta_{(1)}(h^c_n)=&\sum_{p+q+r=n}
	(-1)^{r(q-1)}\mu^c_{p+r+1}\bc {p+1} h^c_{q}+\sum_{p+q+r=n}(-1)^{(r+1)(q-1)} h^c_{p+r+1}\bc {p+1}\mu^c_{q},\\	
	\Delta_{(1)}( \d^c_n)=&\sum_{p+q+r=n}(-1)^{(r+1)(q-1)} \d^c_{p+r+1}\bc {p+1}\mu^c_{q}
	+\sum_{q+r=n}(-1)^{r(q-1)}\mu^c_{r+1}\bc {1}\d^c_{q}\\
	&+\sum_{p+q+r=n\atop p\geq1}(-1)^{r(q-1)}\mu^c_{p+r+1}\bc {p+1} h^c_{q}.
\end{align*}
\end{prop}

\subsection{Coderivations of a cofree $\Agd^{\ac}$-coalgebra}
Let $\ic$ be a ns cooperad and $C$ a $\ic$-coalgebra. 
Recall that a \textit{coderivation} on $C$ is a linear map $d_C\colon C \to C$ satisfying 
\[\Delta_C\circ d_C = \left(\bigoplus\Id_\ic \otimes (\Id_C^{\ot i}\otimes d_C \otimes \Id_C^{\ot n-i})\right)  \circ \Delta_C. \]
Denoted by $\Coder(C)$ the space of all coderivations on $C$. 
On the coderivations of the cofree $\ic$-coalgebra $\ic(V)$, we have the following characterization.

\begin{prop}\cite[Proposition 6.3.8]{LV}\label{3-prop-def of coderivation on cofree coalgebra}
Let $\ic$ be a ns cooperad and $\ic(V)$ a cofree $\ic$-coalgebra for a vector space $V$. Then there is an isomorphism of vector spaces
\[ \Coder(\ic(V))\cong \Hom(\ic(V), V).\]
To be precise,  any coderivation $d$ on $\ic(V)$ is completely determined by its projection onto the space of the cogenerators $\proj_V\circ d\colon \ic(V)\to V$.
Conversely, for any  map $\psi\colon \ic(V)\to V$, there exists a unique coderivation $d_\psi$ on $\mathcal{C}(V)$ given by
\[d_\psi([\theta|v_1, \cdots, v_n])\colon=\sum(-1)^\pm[\theta_{(1)}|v_1, \cdots, v_{i-1}, \psi([\theta_{(2)}| v_{i}, \cdots, v_{i+\art(\theta_{(2)})-1}]), v_{i+\art(\theta_{(2)})}, \cdots, v_n]\]
where $\pm=(|v_1|+\cdots+|v_{i-1}|)(|\psi|+|\theta_{(2)}|)$ and  $\Delta_\ic(\theta)=\sum \theta_{(1)}\bc{i}\theta_{(2)}$. 
\end{prop}

\subsection{Homotopy AsGDer triples}
\begin{Def}
	Let $A$ be a graded vector space. An $\Agd_{\infty}$-\textit{algebra structure} on $A$ consists of collections of operations $\{m_n,h_n,\delta_n|n\geq1\}$ with  
	\[m_n\colon A^{\ot n}\ra A,~ |m_n|=n-2,\quad
	h_n\colon A^{\ot n}\ra A,~ |h_n|=n-1,\quad
	\d_n\colon A^{\ot n}\ra A,~ |\d_n|=n-1,\]
	satisfying the following relations
	\begin{align}
		&\sum_{p+q+r=n}(-1)^{p+qr}m_{p+r+1}\ucr {p+1}m_{q}=0,\label{3-def of homotopy m_n}\tag{E3.3.1}\\
		&\sum_{p+q+r=n}(-1)^{p+qr}m_{p+r+1}\ucr {p+1}h_{q}
		+\sum_{p+q+r=n}(-1)^{p+q(r+1)}h_{p+r+1}\ucr {p+1}m_{q}=0,\label{3-def of homotopy h_n}\tag{E3.3.2}\\	
		&\sum_{p+q+r=n\atop p\geq1}(-1)^{p+qr}m_{p+r+1}\ucr {p+1}h_{q}
		+\sum_{q+r=n}(-1)^{qr}m_{r+1}\ucr {1}\delta_{q}\label{3-def of homotopy delta_n}\tag{E3.3.3}\\
		&+\sum_{p+q+r=n}(-1)^{p+q(r+1)}\d_{p+r+1}\ucr {p+1}m_{q}=0,\nonumber	
	\end{align}
for $n\geq1$, with $-\ucr i-$  defined as in Equations \eqref{1-1-def-partial composition of map}.
The quadruple $(A, m_n, h_n , \delta_n)$ is called an \textit{$\Agd_\infty$-algebra}, or a \textit{homotopy AsGDer triple}.
\end{Def}
Observe that a graded vector space $A$ equipped with the operations $\{m_n\mid n\ge 1\}$ satisfying the relations \eqref{3-def of homotopy m_n}
is just an $A_\infty$-algebra. 

\begin{exm}
Let $(A, m_n)$ be an $A_\infty$-algebra. 
Fixing an element $a\in A$ with $|a|=0$ and $m_1(a)=0$, 
we define  $h_n=0$ and $\delta_n\colon A^{\ot n}\ra A$ as follows:
	\[\delta_n(b_1,\cdots,b_n) \colon=(-1)^{n+1}m_{n+1}(a,b_1,\cdots,b_n) \quad \text{for}\quad b_1,\cdots,b_n\in A. \]
Since $h_n=0 $ for all $n\geq 1$, the relations \eqref{3-def of homotopy h_n} are obviously satisfied.
For convenience, we denote $\delta_n$ by $(-1)^{n+1}m_{n+1}\ucr 1 a$, which means putting the element $a$ in the first position of linear map $m_{n+1}$. Clearly, we have $m_1\ucr 1 a=0$ since $m_1(a)=0$. Then for $n\geq 1$, we have
\begin{align*}
		&\sum_{p+q+r=n\atop p\geq1}(-1)^{p+qr}m_{p+r+1}\ucr {p+1}h_{q}
		+\sum_{q+r=n\atop q\geq1}(-1)^{qr}m_{r+1}\ucr {1}\delta_{q}
		+\sum_{p+q+r=n}(-1)^{p+q(r+1)}\d_{p+r+1}\ucr {p+1}m_{q}\\
		&=\sum_{q+r=n\atop q\geq1}(-1)^{qr}m_{r+1}\ucr {1}\delta_{q}
		+\sum_{p+q+r=n}(-1)^{p+q(r+1)}\d_{p+r+1}\ucr {p+1}m_{q}\\
		&=\sum_{q+r=n\atop q\geq1}(-1)^{qr+q+1}m_{r+1}\ucr {1}(m_{q+1}\ucr 1a)
		+\sum_{p+q+r=n}(-1)^{q(r+1)+r}(m_{p+r+2}\ucr 1a)\ucr {p+1}m_{q}\\
		&=\Big(\sum_{q+r=n\atop q\geq1}(-1)^{qr+q+1}m_{r+1}\ucr {1}m_{q+1} +\sum_{p+q+r=n}(-1)^{q(r+1)+r}m_{p+r+2}\ucr {p+2}m_{q}  \Big)\ucr 1a\\
		&=\Big(\sum_{p+q+r=n+1\atop p=0,q\geq2}(-1)^{(q-1)r+q}m_{p+r+1}\ucr {p+1}m_q+
		\sum_{p+q+r=n+1\atop p\geq1}(-1)^{q(r+1)+r}m_{p+r+1}\ucr {p+1} m_q\Big)\ucr 1a\\
		&=\Big(\sum_{p+q+r=n+1\atop p=0}(-1)^{(q-1)r+q}m_{p+r+1}\ucr {p+1}m_q+
		\sum_{p+q+r=n+1\atop p\geq1}(-1)^{q(r+1)+r}m_{p+r+1}\ucr {p+1} m_q\big)\ucr 1a\\	
		&=\big(\sum_{p+q+r=n+1}(-1)^{rq+r+q}m_{p+r+1}\ucr {p+1}m_q\Big)\ucr 1a\\
		&=(-1)^{n+1}\Big(\sum_{p+q+r=n+1}(-1)^{p+rq}m_{p+r+1}\ucr {p+1}m_q\Big)\ucr 1a=0.		
	\end{align*}
Therefore, the relations \eqref{3-def of homotopy delta_n} hold for all $n\ge 1$, and $(A, m_n, h_n, \delta_n)$ is an  $\Agd_{\infty}$-algebra.	
\end{exm}


\begin{thm}\label{3-thm-Agd_infty str=sq zero coder}
 	Let $A$ be a graded vector space.	An $\As_{\infty}$-algebra structure on  $A$ is equivalent to a square-zero coderivation on the cofree $\As^{\ac}$-coalgebra $\As^{\ac}(A)$.
\end{thm}
\begin{proof}
	Let $\f$ be a square-zero coderivation on the  $\As^{\ac}$-coalgebra $\As^{\ac}(A)$.
	By Proposition \ref{3-prop-def of coderivation on cofree coalgebra}, $\f$ is completely determined by the composite
	\[\As^{\ac}(A)\overset{\f}{\longrightarrow}\As^{\ac}(A)\overset{\pi}{\longrightarrow}A,\]
	where $\pi=\varepsilon_A: \As^{\ac}(A)\twoheadrightarrow \I(A)\cong A$.
	Denote 
	\begin{align*}
	\bar{m}_n:=\pi\c \f|_{\mu^c_n\ot A^{\ot n}},  \quad
	\bar{h}_n:=\pi\c \f|_{h^c_n\ot A^{\ot n}}, \quad
	\bar{\d}_n:=\pi\c \f|_{\d^c_n\ot A^{\ot n}}, 
	\end{align*}
	and we define the operations  $\{m_n,h_n,\delta_n|n\geq1\}$ as follows
	\begin{align}\label{3-eq-proof of Agd_infty str=sq zero coder}\tag{E3.3.4}
	\begin{split}
	m_n(a_1\otimes\cdots\ot a_n)&~:=~\bar{m}_n(\mu^c_n\ot a_1\cdots\ot a_n),\\
	h_n(a_1\otimes\cdots\ot a_n)&~:=~\bar{h}_n(h^c_n\ot a_1\cdots\ot a_n),\\
	\d_n(a_1\otimes\cdots\ot a_n)&~:=~\bar{\d}_n(\d^c_n\ot a_1\cdots\ot a_n),
	\end{split}
		\end{align}
	for all $a_1,\cdots, a_n\in A$. Since $|\bar{m}_n|=|\bar{h}_n|=|\bar{\d}_n|=-1$, $|\mu^c_n|=n-1$, $|h^c_n|=|\d^c_n|=n$, we have $|m_n|=n-2$, $|h_n|=|\d_n|=n-1$.
	
	Consider the restriction of $\pi\c \f\c \f$ on $\As^{\ac}(n)\ot A^{\ot n}$ for each $n\geq1$, and by $\varphi^2=0$, we have 
	\begin{align*}
		0&=\pi\c \f\c \f([\mu^c_n| a_1 ,\cdots, a_n])\\
		&=\sum_{p+q+r=n}(-1)^{r(q-1)+\sigma}\pi\c \f\big([\mu^c_{p+r+1}| a_1\cdots, a_p,
		\bar{m}_q([\mu^c_{q}| a_{p+1},\cdots, a_{p+q}]),\cdots, a_{n}]
		\big)\\
		&=\sum_{p+q+r=n}(-1)^{r(q-1)+\sigma}\bar{m}_{p+r+1}\big([\mu^c_{p+r+1}| a_1\cdots, a_p,
		\bar{m}_q([\mu^c_{q}| a_{p+1},\cdots, a_{p+q}]),\cdots, a_{n}]\big)\\
		&=\sum_{p+q+r=n}(-1)^{r(q-1)+\sigma}m_{p+r+1}\big(a_1\cdots\ot a_p\otimes
		m_q( a_{p+1}\cdots\ot a_{p+q})\ot\cdots\ot a_{n}\big)\\
		&=\sum_{p+q+r=n}(-1)^{p+qr}(m_{p+r+1}\ucr{p+1}m_{q})(a_1\ot\cdots\ot a_n),
	\end{align*}
	where $\sigma=(|a_1|+|a_2|\cdots+|a_p|)q+p+r$. Therefore, we get  relations \eqref{3-def of homotopy m_n}.

	Next, we show that the relations \eqref{3-def of homotopy h_n} hold. In fact, by $\pi\c \f\c \f([h^c_n|a_1,\cdots, a_n])=0$, we have
	\begin{align*}
		0&=\pi\c \f\c \f([h^c_n|a_1,\cdots, a_n])\\
		&=\sum_{p+q+r=n}(-1)^{r(q-1)+\xi}\pi\c \f\big([\mu^c_{p+r+1}| a_1\cdots, a_p,
		\bar{h}_q([h^c_{q}| a_{p+1}\cdots, a_{p+q}]),\cdots, a_{n}]
		\big)\\
		&+\sum_{p+q+r=n}(-1)^{(r+1)(q-1)+\xi'}\pi\c \f\big([h^c_{p+r+1}| a_1,\cdots,a_p,
		\bar{m}_q([\mu^c_{q}| a_{p+1},\cdots,a_{p+q}]),\cdots, a_{n}]
		\big)\\
		=&\sum_{p+q+r=n}(-1)^{r(q-1)+\xi}\bar{m}_{p+r+1}\big([\mu^c_{p+r+1}| a_1\cdots, a_p,
		\bar{h}_q([h^c_{q}| a_{p+1}\cdots, a_{p+q}]),\cdots, a_{n}]\big)\\
		&+\sum_{p+q+r=n}(-1)^{(r+1)(q-1)+\xi'}\bar{h}_{p+r+1}([h^c_{p+r+1}| a_1,\cdots,a_p,
		\bar{m}_q([\mu^c_{q}| a_{p+1},\cdots,a_{p+q}]),\cdots, a_{n}]
		\big)\\
		=&\sum_{p+q+r=n}(-1)^{r(q-1)+\xi}m_{p+r+1}\big(a_1\cdots\ot a_p\otimes
		h_q(a_{p+1}\cdots\ot a_{p+q})\ot\cdots\ot a_{n}\big)\\
		&+\sum_{p+q+r=n}(-1)^{(r+1)(q-1)+\xi'}h_{p+r+1}(a_1\cdots\ot a_p\otimes
		m_q( a_{p+1}\cdots\ot a_{p+q})\ot\cdots\ot a_{n}\big)\\
		=&\sum_{p+q+r=n}(-1)^{p+qr}(m_{p+r+1}\ucr {p+1}h_{q})(a_1\ot\cdots\ot a_n)\\
		&+\sum_{p+q+r=n}(-1)^{p+q(r+1)}(h_{p+r+1}\ucr {p+1}m_{q})(a_1\ot\cdots\ot a_n),	
	\end{align*}
	where $\xi=(|a_1|+|a_2|\cdots+|a_p|)(q+1)+p+r$, and  $\xi'=(|a_1|+|a_2|\cdots+|a_p|)q+p+r+1$. 
	
	Finally, we check the relations \eqref{3-def of homotopy delta_n}. By $\pi\c \f\c \f(\d^c_n\ot a_1\cdots\ot a_n)=0$, we have
	\begin{align*}
		0=&\pi\c \f\c \f([\d^c_n| a_1,\cdots, a_n])\\
		=&\sum_{p+q+r=n\atop p\geq1}(-1)^{r(q-1)+\xi}\bar{m}_{p+r+1}\big([\mu^c_{p+r+1}| a_1\cdots, a_p,
		\bar{h}_q([h^c_{q}| a_{p+1},\cdots, a_{p+q}]),\cdots, a_{n}]\big)\\
		&+\sum_{q+r=n}(-1)^{r(q-1)+r}\bar{m}_{r+1}\big([\mu^c_{r+1}|
		\bar{\d}_q([\d^c_{q}| a_{1},\cdots, a_{p}]),\cdots, a_{n}]
		\big)\\
		&+\sum_{p+q+r=n}(-1)^{(r+1)(q-1)+\xi'}\bar{\d}_{p+r+1}\big([\d^c_{p+r+1}| a_1,\cdots, a_p,\bar{m}_q([\mu^c_{q}, a_{p+1},\cdots, a_{p+q}]),\cdots, a_{n}]
		\big)\\
		=&\sum_{p+q+r=n\atop p\geq1}(-1)^{p+qr}~(m_{p+r+1}\ucr {p+1}h_{q})(a_1\ot\cdots\ot a_n)
		+ \sum_{q+r=n}(-1)^{qr}(m_{r+1}\ucr 1\d_q)(a_1\ot\cdots\ot a_n)\\                                      
		&+\sum_{p+q+r=n}(-1)^{p+q(r+1)}~(\d_{p+r+1}\ucr {p+1}m_{q})(a_1\ot\cdots\ot a_n),	
	\end{align*}
	where $\xi=(|a_1|+|a_2|\cdots+|a_p|)(q+1)+p+r$, and  $\xi'=(|a_1|+|a_2|\cdots+|a_p|)q+p+r+1$. Thus, we get the relations \eqref{3-def of homotopy delta_n}.
	
Conversely, let $(A, m_n, h_n, \delta_n)$ be an $\Agd_{\infty}$-algebra.
Consider the operations $\{\bar{m}_n,\bar{h}_n,\bar{\delta}_n|n\geq1\}$ as in \eqref{3-eq-proof of Agd_infty str=sq zero coder}, 
and by Proposition \ref{3-prop-def of coderivation on cofree coalgebra} one can obtain a coderivation $\f$ of $\As^{\ac}(A)$.
From the same calculation to check the relations \eqref{3-def of homotopy m_n}-\eqref{3-def of homotopy delta_n},
it follows that $\pi\circ \varphi\circ \varphi=0$ and therefore $\varphi \circ \varphi=0$.
\end{proof}

\begin{rmk}
For $n=1$,  the relations \eqref{3-def of homotopy m_n}-\eqref{3-def of homotopy delta_n} mean that 
\[m_1\c m_1=0,\quad m_1\c h_1=h_1\c m_1,\quad m_1\c \d_1=\d_1\c m_1. \]
Therefore, $m_1$ is a degree $-1$ differential of $A$ and $h_1,\d_1$ commutes with $m_1$. For $n=2$, we have
\begin{align*}
	&m_1\ucr 1 m_2-m_2\ucr 1m_1-m_2\ucr 2m_1=0,\\
	&h_1\ucr 1m_2-m_2\ucr 1h_1-m_2\ucr 2h_1=-m_1\ucr 1h_2-h_2\ucr 1m_1-h_2\ucr 2 m_1,\\
	&\d_1\ucr 1m_2-m_2\ucr 1\d_1-m_2\ucr 2h_1=-m_1\ucr 1\d_2-\d_2\ucr 1m_1-\d_2\ucr 2 m_1.	
\end{align*}
Therefore, $m_1$ is a derivation of $m_2$, $h_1$ is a derivation of $m_2$ up to homotopy, and $\d_1$ is a generalized derivation of $m_2$ up to homotopy.
\end{rmk}

\section{Cohomology and Formal Deformation of AsGDer Triples}

In this section,  we construct the cohomology theory for AsGDer triples by using the cooperad $\Agd^{\ac}$,
and show the cohomology controls the formal deformation of AsGDer triples.

\subsection{Coderivation space on the cofree $\Agd^{\ac}$-coalgebra}

Let $V$ be a graded vector space concentrated in degree $0$ and $\Agd^{\ac}(V)$ the cofree $\Agd^{\ac}$-coalgebra. 
By Proposition \ref{3-prop-def of Koszul dual cooperad }, we know that 
the degree $n$ component of $\Agd^{\ac}(V)$ is 
\begin{align}\label{4-eq-the degree n component of Agd^ac(V)}\tag{E4.1.1}
\Agd^{\ac}(V)^{n}=(\Bbbk\mu^c_{n+1}\otimes V^{\otimes, n+1}) \oplus 
(\Bbbk h^c_n\otimes V^{\otimes n}) \oplus (\Bbbk\d^c_n\otimes V^{\otimes n}),
\end{align}
for $n\geq 1$ and  $\Agd^{\ac}(V)^{0}=\Bbbk\mu^c_1\otimes V$. 
Denote by $\Coder^{n}(\Agd^{\ac}(V))$ the vector space of all degree $n$ coderivations of $\Agd^{\ac}(V)$.
Then we have $\Coder^{n}(\Agd^{\ac}(V))=0$ for $n\ge 1$ since elements of $\Agd^{\ac}(V)$ have non-negative degrees.
It is well known that
\[\Coder(\Agd^{\ac}(V))\colon=\bigoplus_{ n\in \Z}\Coder^{n}(\Agd^{\ac}(V))\]
is a graded Lie algebra with Lie bracket $[f,g]=f\circ g-(-1)^{mn}g\circ f$ for $f\in\Coder^{n}(\Agd^{\ac}(V))$ and  $g\in\Coder^{m}(\Agd^{\ac}(V))$. Denote the graded vector space $L_{\Agd}(V)$ by
\begin{align*}
	L_{\Agd}^{-n}(V):=\begin{cases}
		\Hom(V^{\ot, n+1},V)\oplus \Hom(V^{\ot n},V)\oplus\Hom(V^{\ot n},V), &{\rm if} ~n\geq 1,\\
		\Hom(V,V), &{\rm if} ~n=0,\\
		0,& {\rm if} ~n<0.
	\end{cases}
\end{align*}
By Proposition \ref{3-prop-def of coderivation on cofree coalgebra} and the Equation \eqref{4-eq-the degree n component of Agd^ac(V)}, 
there are isomorphisms  of vector spaces:
\[L_{\Agd}^{-n}(V)\xrightarrow{\Phi_{-n}} \Hom(\Agd^{\ac}(V)^{n}, V)\xrightarrow{\Psi_{-n}} \Coder^{-n}(\Agd^{\ac}(V)),\]
where
\begin{align*}
	\Phi_{-n}(f)(\mu^c_{n+1}\otimes v_1\cdots\otimes v_{n+1})&=f_1(v_1\cdots\otimes v_{n+1}),\\
	\Phi_{-n}(f)(h^c_{n}\otimes v_1\cdots\otimes v_{n})&=f_2(v_1\cdots\otimes v_{n}),\\
	\Phi_{-n}(f)(\delta^c_{n}\otimes v_1\cdots\otimes v_{n})&=f_3(v_1\cdots\otimes v_{n}),
\end{align*}
for $f=(f_1,f_2,f_3)\in L_{\Agd}^{-n}(V)$, $\Phi_{-n}(f)\in  \Hom(\Agd^{\ac}(V)^{(n)}, V)$,
and $\Psi_{-n}$ is defined as in Proposition \ref{3-prop-def of coderivation on cofree coalgebra}. 
Consider $\Phi\colon=\bigoplus_{n}\Phi_n$ and $\Psi\colon=\bigoplus_{n}\Psi_n$, and we have the following isomorphism of graded vector spaces
\[\Psi\circ \Phi\colon L_{\Agd}(V)\to \Coder(\Agd^{\ac}(V)).\]
It follows that $L_{\Agd}(V)$ admits a graded Lie algebra structure with the Lie bracket $[-,-]_{L_{\Agd}(V)}$ on $L_{\Agd}(V)$ given by
$[(f_1, f_2, f_3), (g_1, g_2, g_3)]_{L_{\Agd}(V)}=(\psi_1, \psi_2, \psi_3)$ with
\begin{equation}\label{4-eq-the gr Lie bracket for L_Agd(A)}\tag{E4.1.2}
	\begin{aligned}
		\psi_1=&~\sum_{j=1}^{n+1}(-1)^{m(j-1)}f_1\ucr jg_1 -(-1)^{mn}\sum_{j=1}^{m+1}(-1)^{n(j-1)}g_1\ucr jf_1,\\
		\psi_2=&~(-1)^n\sum_{j=1}^{n+1}(-1)^{(m-1)(j-1)}f_1\ucr jg_2 + \sum_{j=1}^{n}(-1)^{(j-1)m}f_2\ucr jg_1 \\ 
		&~-(-1)^{mn}\Big((-1)^m\sum_{j=1}^{m+1}(-1)^{(n-1)(j-1)}g_1\ucr jf_2+\sum_{j=1}^{m}(-1)^{(j-1)n}g_2\ucr jf_1\Big),
		\\                   
		\psi_3=&~\sum_{j=1}^{n}(-1)^{(j-1)m}f_3\ucr jg_1+(-1)^n\big(f_1\ucr1g_3+\sum_{j=2}^{n+1}(-1)^{(m-1)(j-1)}f_1\ucr jg_2\big)\\
		&~-(-1)^{mn}\Big(\sum_{j=1}^{m}(-1)^{(j-1)n}g_3\ucr jf_1+ (-1)^m \big(g_1\ucr1f_3+\sum_{j=2}^{m+1}(-1)^{(n-1)(j-1)}g_1\ucr jf_2\big)\Big),                     	
	\end{aligned}
\end{equation}
for $(f_1, f_2, f_3)\in L_{\Agd}^{-n}(V)$ and $(g_1, g_2, g_3)\in L_{\Agd}^{-m}(V)$, 
where $-\ucr i-$ is defined  in Equation \eqref{1-1-def-partial composition of map}.

\begin{lem}
	Let $A$ be a vector space, and $\mu\colon A^{\ot 2}\ra A$, $h,\delta\colon A\ra A$ be linear maps. The pair $(A,\mu)$
	is an associative algebra and  the triple $(A,h,\delta)$ is an AsGDer triple
	 if and only if $(\mu,h,\delta)\in L^1_{\Agd}(A)$ satisfies $[(\mu,h,\delta),(\mu,h,\delta)]_{L_{\Agd}(A)}=0$.
\end{lem}

\begin{proof}
	By the definition of $[-,-]_{L_{\Agd}(A)}$ in Equation \eqref{4-eq-the gr Lie bracket for L_Agd(A)},  $[(\mu,h,\delta),(\mu,h,\delta)]_{L_{\Agd}(A)}=0$ if and only if $\mu, h$ and $\delta$ satisfies:
	\begin{align*}
		&\mu\ucr 1\mu=\mu\ucr 2\mu,\\
		&\mu\ucr 1h+\mu\ucr 2h-h\ucr 1\mu=0,\\
		&\mu\ucr 1\delta+\mu\ucr 2h-\delta\ucr 1\mu=0.
	\end{align*}
	it follows that $[(\mu,h,\delta),(\mu,h,\delta)]_{L_{\Agd}(A)}=0$ if and only if $\mu$ is an associative product, $h$ is a derivation and $\delta$ is a generalized $h$-derivation.	
\end{proof}

\subsection{The AsGDer complex}
Let $(A,\mu)$ be an associative algebra and $(A,h,\delta)$ be an AsGDer triple.
Then $(\mu,h,\delta)\in L^{-1}_{\Agd}(A)$ satisfies $[(\mu,h,\delta),(\mu,h,\delta)]_{L_{\Agd}(V)}=0$.  
Denote the spaces
\begin{itemize}
	\item $\C^n_{\gder}(A, A)\colon =0$ for $n\le 0$,
	\item ${\C}^1_{\gder}(A, A)\colon=L^0_{\Agd}(A)=\Hom(A,A)$,
	\item $\C^n_{\gder}(A, A)\colon=L^{1-n}_{\Agd}(A)=\Hom(A^{\otimes n},A)\oplus\Hom(A^{\otimes, n-1},A)\oplus\Hom(A^{\otimes, n-1},A)$ for $n>1$,
\end{itemize} 
and the linear map 
\begin{align*}
\partial\colon C^n_{\gder}(A,A)&\to C^{n+1}_{\gder}(A,A),\\
\partial(f)\colon&=(-1)^{n-1}[(\mu,h,\delta),~f]_{L_{\Agd}(A)}.
\end{align*}
By $[(\mu,h,\delta),(\mu,h,\delta)]_{L_{\Agd}(A)}=0$, we have  
\begin{align*}
	\partial^2(f)&=[(\mu,h,\delta),[(\mu,h,\delta),f]_{L_{\Agd}(A)}]_{L_{\Agd}(A)}\\
	&=[[(\mu,h,\delta),(\mu,h,\delta)]_{L_{\Agd}(A)},f]_{L_{\Agd}(A)}-[(\mu,h,\delta),[(\mu,h,\delta),f]_{L_{\Agd}(A)}]_{L_{\Agd}(A)}\\
	&=-[(\mu,h,\delta),[(\mu,h,\delta),f]_{L_{\Agd}(A)}]_{L_{\Agd}(A)}.
\end{align*}
It follows that $\partial^2=0$ and $\partial$ is a differential. Therefore, we obtain a complex $(\C_{\gder}^\bullet(A, A), \partial^\bullet)$.

\begin{Def}
Retain the above notation. Then the complex $(\C_{\gder}^\bullet(A, A), \partial^\bullet)$
is called the \textit{AsGDer complex} for the AsGDer triple $(A, h, \delta)$, and its $n$-th cohomology group
is called the \textit{$n$-th AsGDer cohomology group} of $(A, h, \delta)$, denoted by $H^n_{\gder}(A, A)$, for $n \geq 1$.
\end{Def}

\begin{rmk}
For $f\in \C^1_{\gder}(A,A)=\Hom(A,A)$,  
\begin{align}\label{4-def of coboundary map of degree 1}\tag{E4.2.1}
	\partial(f)=&\big(\mu\ucr 1f+\mu\ucr 2f-f\ucr 1\mu,~h\ucr 1f-f\ucr 1h,~\delta\ucr 1 f-f\ucr 1\delta\big).
\end{align}
For $f=(f_1,f_2,f_3)\in C^n_{\gder}(A,A)$ ($n\ge 2)$,  
\begin{equation}\label{4-def of coboundary map n}\tag{E4.2.2}
	\begin{aligned}
		\partial(f)=\Big(&\mu\ucr 2f_1+\sum_{i=1}^{n}(-1)^if_1\ucr i\mu+(-1)^{n+1}\mu\ucr 1f_1,\\
		&\mu\ucr{2}f_2+(-1)^n\mu\ucr{1}f_2+\sum_{i=1}^{n-1}(-1)^i f_2\ucr{i}\mu+(-1)^{n-1}\big(h\ucr{1}f_1-\sum_{i=1}^{n}f_1\ucr{i}h\big),\\  
		&\mu\ucr{2}f_2+(-1)^n\mu\ucr{1}f_3+\sum_{i=1}^{n-1}(-1)^i f_3\ucr{i}\mu+(-1)^{n-1}\big(\delta\ucr{1}f_1-f_1\ucr{1}\delta-\sum_{i=2}^{n}f_1\ucr{i}h\big)\Big).\\
	\end{aligned}	
\end{equation}
Denote by ${\rm Der}(A)$ the space of all derivations of the algebra $A$, from \eqref{4-def of coboundary map of degree 1}, we have $H^1_{\gder}(A,A)=\{f\in {\rm Der}(A) ~|~f\c h=h\c f~ {\rm and}~ f\c \delta=\delta\c f \}$.

\end{rmk}

Next, we consider the cohomology of an AsGDer triple with coefficients in it's module.
\begin{Def}
	Let $(A,h,\delta)$ be an AsGDer triple.
	An AsGDer-module over $(A, h, \delta)$ consists of a triple $(M,h_M,\d_M)$ in which $M$ is an $A$-$A$-bimodule and $h_M,\d_M : M \rightarrow M$ are  linear maps satisfying
	\begin{align}\label{4-def of module of AsGDer}\tag{E4.2.3}
		\begin{split}
			h_M ( a m ) &=h (a) m + a h_M (m),\\
			h_M (m a )& = h_M (m) a + m h (a),\\
			\delta_M ( a m ) &=\delta (a) m + a h_M (m),\\
			\delta_M ( ma ) &=\delta_M (m) a + m h (a),
		\end{split}	
	\end{align}
	for all	$a\in A$ and $m\in M$.	
\end{Def} 	

It is easily seen that the $(A,h,\delta)$ is a module over itself.  
Suppose that $M$ is an $A$-$A$-bimodule over an associative algebra $A$.
Recall that the semi-direct product $A\ltimes M$ is an
associative algebra with $A\ltimes M=A\oplus M$ as a vector space and the multiplication given by 
\begin{align*}
	(a, m) \cdot (b, n) = (ab, an + mb).
\end{align*}
On AsGDer triples, we have a similar result.

\begin{lem}\label{4-lem-semi-direct}
Let $(A, h, \delta)$ be an AsGDer triple  and $(M, h_M,\d_M)$ be an AsGDer-module over $(A, h, \delta)$. 
Then $\d \oplus \delta_M\colon A\ltimes M \to A\ltimes M$ is a generalized $h\oplus h_M$-derivation.
\end{lem}
\begin{proof}
By \cite[Proposition 3]{DM}, we know that $h  \oplus h_M$ is a derivation.
For any $(a, m), (b, n)\in A\ltimes M$, we have
	\begin{align*}
		&(\d  \oplus \d_M) ( (a, m) \cdot (b, n) ) \\=~& ( \delta (ab) , \d_M (an) + \d_M (mb)  ) \\
		=~& (\d (a) b + a h (b), \d (a)n + a h_M (n) + \d_M (m) b + m h (b) ) \\
		=~& ( \d (a) b , \d (a) n + \d_M (m) b) ~+~ ( a h(b) , a h_M (n) + m h (b) ) \\
		=~& (\d (a), \d_M (m)) \cdot (b, n) ~+~ (a, m) \cdot (h(b) , h_M (n) ) \\
		=~& (\d \oplus \d_M )(a, m) \cdot (b, n) ~+~ (a, m) \cdot (h \oplus h_M) (b, n).
	\end{align*}
	It follows that $\d  \oplus \d_M \colon  A \oplus M \rightarrow A \oplus M$ is a generalized $h\oplus h_M$-derivation. 
\end{proof}

Let $(M, h_M,\d_M)$ be an AsGDer-module over an  AsGDer triple $(A, h, \delta)$. 
By Lemma \ref{4-lem-semi-direct}, $(A\ltimes M, h \oplus h_M,\d \oplus \delta_M)$ is also an AsGDer triple, called the semi-direct of
$(A, h, \delta)$ and $(M, h_M,\d_M)$.
Therefore, we get a cochain complex $(C_{\gder}^\bullet(A\ltimes M, A\ltimes M),\partial_{A\ltimes M})$.
Define 
\[\C^n_{\gder}(A,M)\colon=\begin{cases}
0, & n\le 0,\\
\Hom(A, M), & n=1,\\
\Hom(A^{\otimes n},M)\oplus\Hom(A^{\otimes, n-1},M)\oplus\Hom(A^{\otimes, n-1},M), & n\ge 2.
\end{cases}\]
Observe that $\C^\bullet_{\gder}(A,M)$ is a subspace of $ C^\bullet_{\gder}(A\ltimes M, A\ltimes M)$.
It is easy to verify that 
\[\partial_{A\ltimes M}(\C^\bullet_{\gder}(A, M))\subseteq \C^{\bullet+1}_{\gder}(A, M).\]
Consequently, we get a cochain complex $(\C^\bullet_{\gder}(A,M),\partial)$, 
where $\partial$ is the restriction of the linear map $\partial_{A\ltimes M}$ on $\C^\bullet_{\gder}(A, M)$.

\begin{Def}
Let $(A, h, \delta)$ be an  AsGDer triple and $(M, h_M,\d_M)$ an AsGDer-module over  $(A, h, \delta)$. 
The cochain complex $(\C^\bullet_{\gder}(A,M),\partial)$ is called the \textit{AsGDer complex for $(A, h, \delta)$ with coefficients in the AsGDer-module $(M, h_M,\d_M)$}, and 
the $n$-th cohomology group of $(\C^\bullet_{\gder}(A,M),\partial)$  is called
the \textit{$n$-th AsGDer cohomology group of an AsGDer triple $(A,h,\delta)$ with coefficients in the AsGDer-module $(M, h_M,\d_M)$},
denoted by $H^n_{\gder}(A, M)$, for $n \geq 1$.
\end{Def}

\subsection{Connection with AssDer cohomology groups}
Let $A$ be an associative algebra and $M$ be an $A$-$A$-bimodule. 
Recall that the Hochschild cochain complex of $A$ with coefficients in $M$ is $(\C_{\h}^\bullet(A, M), d_{\rm Hoch})$ with 
$$\C^n_{\h}(A, M)=
\begin{cases}
0, & n=0,\\
\Hom(A^{\ot n}, M), & n\ge 1,
\end{cases}$$ and the differential $d_{\rm Hoch}\colon \C_{\h}^n(A, M)\to \C_{\h}^{n+1}(A, M)$ $(n\ge 1)$ given by 
\begin{equation*}
	\begin{aligned}\label{4-def-hoch differential}
		d_{\rm Hoch}(f)(a_1,a_2,\cdots,a_{n+1}) =& a_1f(a_2,\cdots,a_{n+1}) 
		+\sum_{i=1}^{n}(-1)^{i}f(a_1,\cdots,a_{i-1},a_ia_{i+1}, a_{i+2}, \cdots,a_{n+1})\\
		&+ (-1)^{n+1}f(a_1,\cdots,a_n)a_{n+1}.
	\end{aligned}
\end{equation*}
for all $f\in\C_{\h}^{n}(A,M)$, $a_1,...,a_{n+1}\in A$ and $n\ge 1$. The corresponding cohomology groups are denoted by $ H^n_{\rm Hoch}(A, M)$.

Note that the cochain complex $(\C_{\h}^\bullet(A, M), d_{\rm Hoch})$ is different with the classical Hochschild complex 
at the term of degree $0$, since $\C_{\h}^0(A, M)=0$ rather than $M$. So the first Hochschild cohomology group $H_{\h}^1(A, M)=\Der(A, M)$, 
rather than the space of outer derivations of $A$ with values in $M$.

Recall that an associative algebra $A$ together with a derivation $h\colon A \to A$ is called an \textit{AssDer pair}, denoted by $(A, h)$, 
and a pair $(M, h_M)$ is called a module over $(A, h)$ if $M$ is an $A$-$A$-bimodule and the linear map $h_M\colon M\to M$ satisfies 
\[h_M(am)=h(a)m+ah_M(m)\quad {\rm and} \quad h_M(ma)=h_M(m)a+mh(a) \]
for all $a\in A, m\in M$, see \cite{DM} for more details. 
In the same paper, the authors constructed the cohomology of the AssDer pairs. 
	
\begin{Def}[{\cite[Section 2.1]{DM}}]
Let $(A, h)$ be an AssDer pair and $(M, h_M)$ a module over $(A, h)$.
Define
\[\C^n_{{\rm AssDer}}(A, M)\colon=\begin{cases}
	0, & n\le 0,\\
	\Hom(A, M), & n=1,\\
	\Hom(A^{\otimes n}, M)\oplus\Hom(A^{\otimes, n-1}, M), & n\ge 2.
\end{cases}\]
and the differential $\partial' \colon \C^n_{\mathrm{AssDer}} (A, M) \rightarrow \C^{n+1}_{\mathrm{AssDer}} (A, M) $ 
\begin{align*}
	\begin{cases}
		\partial'(f)= (d_{\mathrm{Hoch}}(f), -\Delta(f) ), & \text{ for } f \in \C^1_{\mathrm{AssDer}} (A, M) = \mathrm{Hom} (A, M),\\
		\partial' (f_n, \overline{f}_n) = ( d_{\mathrm{Hoch}}(f_n),  d_{\mathrm{Hoch}}(\overline{f}_n)+(-1)^n  \Delta(f_n)), & \text{ for } (f_n, \overline{f}_n) \in \C^n_{\mathrm{AssDer}} (A, M), 
	\end{cases}
\end{align*}
where the linear map $\Delta\colon \mathrm{Hom}(A^{\otimes n}, M) \rightarrow \mathrm{Hom}(A^{\otimes n}, M) $ is given by
$$\Delta f = \sum_{i=1}^n f \circ (\mathrm{id}^{\otimes i-1} \otimes  h \otimes \mathrm{id}^{\otimes n-i}) - h_M \circ f.$$
The cochain complex $(\C^\bullet_{\mathrm{AssDer}} (A, M), \partial')$ is the  \textit{AssDer complex of $(A, h)$ with coefficients in its module $(M, h_M)$}, and the corresponding cohomology groups are denoted by $ H^n_{\rm AssDer}(A,M)$, for $n \geq 1$.
\end{Def}
 
 \begin{rmk}
 Observe that $\Delta\colon \C_{\h}^\bullet(A, M) \to \C_{\h}^\bullet(A, M)$ 
 is a chain map and the AssDer complex $(\C^\bullet_{\mathrm{AssDer}} (A, M), \partial')$ is exactly isomorphic to the negative shift of the mapping cone of the chaim map $\Delta$. 
\end{rmk}	
	
Based on the above observation, we have the following result.
\begin{prop}\label{4-prop-relations-cohomology groups-AssDer}
Let  $\C^\bullet_{\mathrm{AssDer}} (A, M)$ be the AssDer complex of the AssDer pair $(A, h)$ with coefficients in the $(A, h)$-module $(M, h_M)$.
Then there is a long exact sequence
\begin{align*}
\begin{split}
0 & \rightarrow H^1_{\mathrm{AssDer}} (A, A)\rightarrow \Der(A, M) \rightarrow \Der(A, M)
\rightarrow  H^2_{\mathrm{AssDer}} (A, A)\rightarrow H^2_{\mathrm{Hoch}} (A, A) \\
&\to \cdots\rightarrow H^{n-1}_{\mathrm{Hoch}} (A, A)\rightarrow H^n_{\mathrm{AssDer}} (A, A)\rightarrow H^n_{\mathrm{Hoch}} (A, A)\rightarrow \cdots.
\end{split}
\end{align*}		
\end{prop}

Now, we consider the relations between AssDer complex and AsGDer complex.
Let $(A, h, \delta)$ be an AsGDer triple amd $(M, h_M, \delta_M)$ a module over $(A, h, \delta)$.
Obviously, $(A, h)$ is an AssDer pair and $(M, h_M)$ is a module over $(A, h)$.	
Suppose that $\C^\bullet_{\gder} (A, M)$ is the AsGDer complex of $(A, h, \delta)$ with coefficients in $(M, h_M, \delta_M)$ 
and $\C^\bullet_{\mathrm{AssDer}} (A, M)$ is the AssDer complex of AssDer pair $(A,h)$ with coefficients in $(M, h_M)$. 
Then we have a chain map 
$\Theta^\bullet\colon \C^\bullet_{\gder} (A, A)\to \C^\bullet_{\mathrm{AssDer}} (A, A)$ defined by:
\[\begin{cases}
	\Theta^1=\Id, & \\
	\Theta^n(f_1, f_2, f_3)= (f_1, f_2), & n\ge 2.
\end{cases}\]
Clearly, $\Theta^\bullet$ is an epimorphism of cochain complexes. Therefore, we have a short exact sequence of cochain complexes:	
\begin{align*}
	\xymatrix{
		0\ar[r]& \ka^\bullet(A,M) \ar[r]
		& \C^\bullet_{\gder} (A, M) \ar[r] ^{\Theta^\bullet\quad} & \C^\bullet_{\mathrm{AssDer}} (A, M)  \ar[r] &0,}
\end{align*}
where the cochain complex $\ka^\bullet(A,M)$ is the kernel of the chain map $\Theta^\bullet$. To be precise, 
\[\ka^n(A,M)=\begin{cases}
	0, & n\le 1,\\
	\Hom(A^{\otimes n-1}, M)  & n\ge 2,
\end{cases}\]
and the differential $\partial_{\ka}^n :\ka^n(A, M)\to \ka^{n+1}(A, M)$ is given by
\begin{equation*}
	\partial^n_{\ka}(f)(a_1, a_2, \cdots, a_{n}) =
	\sum_{i=1}^{n-1}(-1)^{i}f(a_1, \cdots, a_{i-1}, a_ia_{i+1}, a_{i+2}, \cdots, a_{n})
	+ (-1)^{n}f(a_1, \cdots, a_{n-1})a_{n}.
\end{equation*}
for any $f\in\ka^n(A, M)$.

\begin{prop}\label{4.3-prop-rela-cohomology groups-AsGDer-1}
Retain the above notation. Then there exists a long exact sequence
	\begin{align*}
		\begin{split}
	0 & \rightarrow H^1_{\gder} (A, M)\rightarrow H^1_{\mathrm{AssDer}} (A, M)\rightarrow H^2(\ka^\bullet(A, M))
	\rightarrow H^2_{\gder} (A, M) \rightarrow H^2_{\mathrm{AssDer}} (A, M) \\
	&\to \cdots \to H^{n-1}_{\mathrm{AssDer}} (A, M) \rightarrow H^{n}(\ka^\bullet(A,M))\rightarrow H^n_{\gder} (A, M)\rightarrow H^n_{\mathrm{AssDer}} (A, M)\rightarrow \cdots.	
		\end{split}		
	\end{align*}
\end{prop}

Next we give a characterization for the cohomology groups of $\ka^\bullet(A, M)$.
Assume that $A$ is a non-unital associative algebra. 
One can consider the augmented algebra $\Lambda=\Bbbk \1 \oplus A$ with the augmented ideal $A$ and the identity element $\1$. 
Clearly, the category of right $A$-modules is isomorphic to the category
of right unitary modules over $\Lambda$. 
For brevity, we use $[a_1\mid \cdots\mid a_n]$ to denote the homogeneous element
$a_1\otimes \cdots \otimes a_n\in A^{\otimes n}$.

\begin{lem}\label{4.3-lem-cohomology group of K(A, M)}
	Let $\Lambda=\Bbbk \1 \oplus A$ be an augmented algebra with augmented ideal $A$,
	and $M$ be a right $\Lambda$-module. 
	Then we have the isomorphism 
	\[\Ext_\Lambda^n(A, M)\cong H^{n+2}(\ka^\bullet(A, M)),\]
	for all $n\ge 0$.
\end{lem}
\begin{proof}
	Being similar to the normalized  resolution introduced in \cite[Section 10.2]{Ma}, we have the free resolution $\mathbf{P}_\bullet \to A$ for the right $\Lambda$-module $A$,  where 
	\[\mathbf{P}_\bullet \colon \cdots \to A^{\otimes, n+1}\otimes \Lambda \xrightarrow{b_n} A^{\otimes n}\otimes \Lambda
	\to \cdots \to A^{\otimes 3}\otimes \Lambda \xrightarrow{b_2} A^{\otimes 2}\otimes \Lambda \xrightarrow{b_1} A\otimes \Lambda \to 0,\]
	and the differential 
	\[b_{n}([a_1\mid\cdots \mid a_{n+1}]\otimes \alpha)=\sum_{i=1}^{n}(-1)^{i-1} 
	[a_1\mid \cdots \mid a_{i}a_{i+1}\mid \cdots \mid a_{n+1}]\otimes \alpha+(-1)^{n}
	[a_1\mid\cdots \mid a_{n}]\ot (a_{n+1}\al)\]
	for all $[a_1\mid\cdots \mid a_{n+1}]\otimes \alpha\in A^{\ot, n+1}\ot \Lambda$. 
	
	Applying the functor $\Hom_\Lambda(-, M)$ to the complex $\mathbf{P}_\bullet$, 
	we get a complex $\Hom_\Lambda(\mathbf{P}_\bullet, M)$ as follows
	\begin{align*}
		0 & \to \Hom_\Lambda(A\ot \Lambda, M) \xrightarrow{b_1^\ast} \Hom_\Lambda(A^{\ot 2}\ot \Lambda, M)\to \Hom_\Lambda(A^{\ot 3}\ot \Lambda, M)\to\cdots \\
		& \to \Hom_\Lambda(A^{\ot, n}\ot \Lambda, M) \xrightarrow{b_n^\ast}\Hom_\Lambda(A^{\ot, n+1}\ot \Lambda, M) \to \cdots.
	\end{align*}
	By the isomorphisms of vector spaces
	\[\Hom_\Lambda(A^{\ot n}\ot \Lambda, M)\cong \Hom(A^{\ot n}, M),\]
	we know that the complex $\Hom_\Lambda(\mathbf{P}_\bullet, M)$ is isomorphic to the  complex
	\[\mathbf{P}^\bullet(A, M)\colon 0 \to \Hom(A, M) \xrightarrow{d_1} \Hom(A^{\ot2}, M) \to \cdots \to \Hom(A^{\ot n}, M)
	\xrightarrow{d_n} \Hom(A^{\ot, n+1}, M)\to \cdots,\]
	where the differential 
	\begin{align*}
		d^n(f)([a_1\mid\cdots\mid a_{n+1}]) =
		\sum_{i=1}^{n}(-1)^{i-1}f([a_1\mid\cdots\mid a_ia_{i+1}\mid\cdots\mid a_{n+1}])
		+ (-1)^{n}f([a_1\mid\cdots\mid a_{n}])a_{n+1}.
	\end{align*}
	for all $a_1\ot\cdots\ot a_{n+1}\in A^{\ot, n+1}$.
	It is easily seen that there is an isomorphism of complexes
	\[\mathbf{P}^\bullet(A, M)\cong \ka^\bullet(A, M)[-2].\]
	Therefore, we have 
	\[\Ext_\Lambda^n(A, M)\cong H^{n+2}(\ka^\bullet(A, M)),\]
	for all $n\ge 0$.
\end{proof}

If $A$ is a unital associative algebra, we can take $\Lambda=A$ and
obtain the same isomorphism in Proposition \ref{4.3-lem-cohomology group of K(A, M)}.
Therefore, we have the following simple observation because $A$ is a projective right $A$-module.

\begin{lem}\label{4.3-lem-cohomology group of K(A, M) for unital A}
	Let $A$ be a unital associative algebra and $M$ a right $A$-module. Then 
	$H^{2}(\ka^{\bullet}(A,M))\cong M$ 
	and $H^{n}(\ka^{\bullet}(A,M))=0$ for all $n\neq 2$.
\end{lem}  

By the long exact sequence in Proposition \ref{4.3-prop-rela-cohomology groups-AsGDer-1} and Lemma \ref{4.3-lem-cohomology group of K(A, M) for unital A}, 
we have the following result.

\begin{prop}\label{4-prop-cohomology group-unital-version}
	Let $(A, h, \delta)$ be an AsGDer triple, where $A$ is a unital associative algebra, and let $(M, h_M, \delta_M)$ be a module over $(A, h, \delta)$.
	Then we have the exact sequence
	\begin{align*}
		0  \to H^1_{\gder} (A, M)\rightarrow H^1_{\mathrm{AssDer}} (A, M)\to M
		\to H^2_{\gder} (A, M) \rightarrow H^2_{\mathrm{AssDer}} (A, M) \to 0
	\end{align*}
	and the isomorphism 
	\[H_{\gder}^n(A, M)\cong H_{{\rm AssDer}}^n(A, M)\]
	for all $n\ge 3$.
\end{prop}

\begin{cor}\label{4-corollary-case-hoch=0}
	Let $(A, h, \delta)$ be an AsGDer triple, where $A$ is a unital associative algebra, and let $(M, h_M, \delta_M)$ be a module over $(A, h, \delta)$.
	\begin{itemize}
		\item [(i)] If $H^2_{\h}(A,M)=0$, then $\dim H^2_{\gder}(A,M)=\dim H^1_{\gder}(A,M)+\dim M $.
		\item [(ii)] If $H^{n-1}_{\h}(A,M)=H^n_{\h}(A,M)=0$ for some $n\geq 3$, then $H^n_{\gder}(A,M)=0$.
	\end{itemize}
	In particular, if $A$ is semisimple as an associative algebra, then $H^n_{\gder}(A, M)=0$ for all $n\geq 3$.
\end{cor}
\begin{proof}
	(i) By Proposition \ref{4-prop-relations-cohomology groups-AssDer} and	$H^2_{\h}(A,M)=0$, we obtain the exact sequence
	\begin{align*}
		0 & \rightarrow H^1_{\mathrm{AssDer}} (A, M)\rightarrow \Der(A,M) \rightarrow \Der(A, M)
		\rightarrow  H^2_{\mathrm{AssDer}} (A, M)\rightarrow 0. 	
	\end{align*}
	Thus $\dim H^2_{\mathrm{AssDer}} (A, M)=\dim H^1_{\mathrm{AssDer}} (A, M)$.
	From the exact sequence in Proposition \ref{4-prop-cohomology group-unital-version}, it follows that
	\[\dim H^2_{\gder} (A, M)= \dim H^1_{\gder} (A, M)+\dim M. \]
	(ii) If $H^{n-1}_{\h}(A,M)=H^n_{\h}(A,M)=0$ for some $n\geq 3$, 
	 we have $H^{n}_{\rm AssDer}(A,M)=0$ by Proposition \ref{4-prop-relations-cohomology groups-AssDer}. From the isomorphism in Proposition \ref{4-prop-cohomology group-unital-version}, it follows that
	\[H^n_{\gder}(A, M)\cong  H^n_{\rm AssDer}(A, M)=0.  \]
\end{proof}

Let $(A, h)$ be an AssDer pair and $(M, h_M)$ a module over $(A, h)$. It is easily seen that $(A, h, h)$ is an AsGDer triple and $(M, h_M, h_M)$ is a module over $(A, h, h)$.
To avoid confusion, we use $\C^\bullet _{\gder} (A, M)_{\delta=h}$ to denote the AsGDer complex of $(A, h, h)$ with coefficients in $(M, h_M, h_M)$
and  $H^n_{\gder} (A, M)_{\delta=h}$ to denote the $n$-th cohomology group, respectively. 
In this situation, we have the chain map $\Phi^\bullet \colon \C^\bullet _{\mathrm{AssDer}} (A, M) \to \C^\bullet _{\gder} (A, M)_{\delta=h}$ given by
\[\begin{cases}
	\Phi^1=\Id, & \\
	\Phi^n(f_1,f_2)=(f_1,f_2,f_2), & n\ge 2,
\end{cases}\]
Clearly, $\Theta^n\Phi^n=\Id$.  It follows that the cochain complex
$\C^\bullet _{\gder} (A, M)_{\delta=h}$ is a direct sum of cochian complexes $\C^\bullet _{\mathrm{AssDer}} (A, M)$ and $\ka^\bullet(A,M)$
when $\delta=h$. 
Therefore, we have the following result.

\begin{cor}\label{4-cor-relation-cohomology group-AsGDer-2}
	Retain the above notation.  Then we have 
	$H^1_{\gder}(A, M)_{\delta=h}=H^1_{\mathrm{AssDer}} (A, M)$ and $H^n_{\gder} (A, M)_{\delta=h}=H^n_{\mathrm{AssDer}} (A, M)\oplus H^n(\ka^\bullet(A, M))$ for $n\geq 2$.
\end{cor} 

\begin{exm}
Let $A=\mathbf{M}_{m}(\Bbbk)$ be the full matrix algebra over $\Bbbk$ of order $m$. 
It is well known that  any derivation $h$ on $A$ is inner, 
that is,  there exists an element $a \in A$ such that $h=[a, -]$,
where $[-, -]$ is the commutator. 
In addition, for any generalized $h$-derivation $\delta$,
there exists an element $b\in A$ such that $\delta(x)=bx+[a,x]$ for all $x\in A$.
Next we calculate the AsGDer cohomology group for 
the AsGDer triple $(A, h, \delta)$. 

For any $f\in \Ker\partial$,  we know that $f$ is a derivation and commutes with $h,\delta$. 
Suppose that $f=[y, -]$ for some $y\in A$.
Then $f\circ h=h\circ f$ and $f\circ \delta=\delta\circ f$ imply that
\begin{align*}
	ya=ay, \quad{\rm and}\quad yb=by.	
\end{align*}
It follows that $y\in Z(a)\cap Z(b)$, where $Z(a)$ and $Z(b)$ are the centerlizer of $a$ and $b$ in $A$, respectively.
Note that $[y,-]=[y',-]$ if and only if $y-y'\in Z(A)=\Bbbk I_m$,
where $Z(A)$ is the center of $A$ and $I_m$ is the identity matrix of order $m$.
Therefore, we get $\dim H^1_{\gder}(A, A)=\dim(Z(a)\cap Z(b))-1$,

Since $A$ is semisimple,  it follows from Corollary \ref{4-corollary-case-hoch=0} that 
\[\dim H^2_{\gder} (A, A)=\dim(Z(a)\cap Z(b))-1+m^2, \quad{\rm and} \quad H^n_{\gder} (A, A)=0 \quad {\rm for}\quad n\geq 3. \]
\end{exm}

\subsection{Formal deformations of AsGDer triples}

In this part, we study the formal deformation of an AsGDer triple, and show the AsGDer cohomology groups
control this formal deformation.

Let $(A, \mu)$ be an associative algebra and $(A,h,\delta)$ be an AsGDer triple. 
Let $\Bbbk[[t]]$ and $A[[t]]$ be the formal power series ring in one variable $t$ with coefficients in $\Bbbk$ and $A$, respectively.
Clearly, $A[[t]]$ is a $\mathbb{K}[[t]]$-module. 
Observe that each $\Bbbk[[t]]$-linear map $\varphi_t\colon A[[t]]\to A[[t]]$ is determined by their restriction to the subset $A$. 
Therefore, we may write $\varphi_t(a)=\sum_{i\geq 0}  \varphi_i(a) t^i$ for any $a\in A$, 
where $\varphi_i\colon A \to A$ is $\Bbbk$-linear map for all $i\ge 0$.

A formal deformation of  $(A, h, \delta)$ means an AsGDer triple $(A_t, h_t, \delta_t)$ over $\Bbbk[[t]]$, such that
$A_t$ is a formal deformation of $A$, and 
\begin{align*}
	h_t = & \sum_{i\geq 0}  h_i t^i, ~~~~~ h_i \in \mathrm{Hom} (A, A) ~~ \text{ and } h_0 = h,\\
	\delta_t = & \sum_{i \geq 0}  \delta_i t^i, ~~~~~ \delta_i \in \mathrm{Hom} (A, A) ~~ \text{ and } \delta_0 = \delta.	
\end{align*}
Similarly, we may write the multiplication $\mu_t$ of $A[[t]]$ as 
\[\mu_t(a, b)=\sum_{i\ge 0}\mu_i(a, b)t^i\]
for $a, b\in A$, where $\mu_i\colon A\otimes A \to A$ is $\Bbbk$-linear map for all $i\ge 0$ and $\mu_0=\mu$ is the multiplication of $A$. 

By easy computing, one can obtain the following  deformation equations for AsGDer triples:
\begin{align}
 \Sum_{i+j=n\atop i, j>0} & \mu_i(\mu_j(a, b), c)-\mu_i(a, \mu_j(b, c))=a\mu_n(b, c)-\mu_n(ab, c)+\mu_n(a, bc)-\mu_n(a, b)c, \label{4.4-eq-defor Eq 1}\tag{DE$_n$-1}\\
 \Sum_{i+j=n\atop i, j>0}& h_i(\mu_j(a, b))-\mu_i(h_j(a), b)-\mu_i(a, h_j(b)) \notag\\
&=\mu_n(a, h(b))+\mu_n(h(a), b)-h(\mu_n(a, b))+h_n(a)b+ah_n(b)-h_n(ab) \label{4.4-eq-defor Eq 2}\tag{DE$_n$-2}\\
\Sum_{i+j=n\atop i, j>0}&\delta_i(\mu_j(a, b))-\mu_i(\delta_j(a), b)-\mu_i(a, h_j(b))\notag\\
&=\mu_n(\delta(a), b)+\mu_n(a, h(b))-\delta(\mu_n(a, b))+\delta_n(a)b+ah_n(b)-\delta_n(ab), \label{4.4-eq-defor Eq 3}\tag{DE$_n$-3}
\end{align}
for all $a, b, c\in A$ and $n \ge 1$. 

As we expected, the AsGDer cohomology controls the formal deformation of AsGDer triples. 

The triple $(\mu_1, h_1,\delta_1)$ is called an \textit{infinitesimal deformation} of the AsGDer triple $(A, h, \delta)$,
which gives an AsGDer triple structure on the quotient space $A[[t]]/(t^2)$. 

We denote $\mathrm{Ob}'_n\colon A^{\ot 3}\to A$ and $\mathrm{Ob}''_n, \ob'''_n\colon A^{\ot 2}\to A$ as follows
\begin{align*}
	\mathrm{Ob}'_n(a, b, c)&=\sum_{i+j=n,\atop i,j>0}\mu_i(\mu_j(a, b), c)-\mu_i(a, \mu_j(b, c)),\\
	\mathrm{Ob}''_n(a, b)&=\sum_{i+j=n,\atop i,j>0}h_i(\mu_j(a, b))-\mu_i(h_j(a), b)-\mu_i(a, h_j(b)),\\
	\ob'''_n(a, b)&=\sum_{i+j=n,\atop i,j>0}\delta_i(\mu_j(a, b))-\mu_i(\delta_j(a), b)-\mu_i(a, h_j(b)).
\end{align*}
Being similar to the case of associative algebras, the formal deformation of an AsGDer will meet so-called obstructions 
in the AsGDer cohomology, which are all deducible from the deformation Equations \eqref{4.4-eq-defor Eq 1}, \eqref{4.4-eq-defor Eq 2} and \eqref{4.4-eq-defor Eq 3}.
By a routine check, we have the following result.

\begin{thm}\label{5-thm-deformation thm}
	Let $(A, h, \delta)$ be an AsGDer triple and $H_{\gder}^n(A,A)$ be the $n$-th AsGDer cohomology group of $(A, h, \delta)$.
	\begin{enumerate}
		\item[(i)] The triple $(\mu_1, h_1, \delta_1)$ is an infinitesimal deformation of $(A, h, \delta)$ if and only if 
		$(\mu_1, h_1, \delta_1)$ is a $2$-cocycle in the AsGDer complex of $A$. In particular, if $H_{\gder}^2(A,A)=0$,
		then $(A, h, \delta)$ has no nontrivial formal deformation.
		
		\item[(ii)] For each $n\ge 2$, if the deformation equations {\rm (DE$_k$-1) }, {\rm (DE$_k$-2) } and {\rm (DE$_k$-3) } hold for $k=1, \cdots, n-1$,
		then $(\ob'_n, \ob''_n, \ob'''_n)$ is a $3$-cocycle in the AsGDer complex, and hence $H_{\gder}^3(A,A)$ is the obstruction cocycle. In particular,
		if $H_{\gder}^3(A,A)=0$, then all the obstructions vanish.
	\end{enumerate}	
\end{thm}

\end{document}